\documentclass[12pt]{amsart}
\usepackage{fullpage}
\usepackage{amssymb,amsmath,stmaryrd,mathrsfs}
\def\definetac{\newif\iftac}    
\ifx\tactrue\undefined
  \definetac
  \ifx\state\undefined\tacfalse\else\tactrue\fi
\fi
\iftac\else\usepackage{amsthm}\fi
\usepackage[all,2cell]{xy}
\usepackage{enumitem}
\usepackage{color}
\definecolor{darkgreen}{rgb}{0,0.45,0} 
\usepackage[pagebackref,colorlinks,citecolor=darkgreen,linkcolor=darkgreen]{hyperref}
\usepackage{mathtools}          
\usepackage{graphics}
\usepackage{braket}             
\let\setof\Set
\usepackage{url}                
\usepackage{xspace}             
\usepackage{cite}               

\makeatletter
\let\ea\expandafter

\def\mdef#1#2{\ea\ea\ea\gdef\ea\ea\noexpand#1\ea{\ea\ensuremath\ea{#2}\xspace}}
\def\alwaysmath#1{\ea\ea\ea\global\ea\ea\ea\let\ea\ea\csname your@#1\endcsname\csname #1\endcsname
  \ea\def\csname #1\endcsname{\ensuremath{\csname your@#1\endcsname}\xspace}}

\DeclareRobustCommand\widecheck[1]{{\mathpalette\@widecheck{#1}}}
\def\@widecheck#1#2{%
    \setbox\z@\hbox{\m@th$#1#2$}%
    \setbox\tw@\hbox{\m@th$#1%
       \widehat{%
          \vrule\@width\z@\@height\ht\z@
          \vrule\@height\z@\@width\wd\z@}$}%
    \dp\tw@-\ht\z@
    \@tempdima\ht\z@ \advance\@tempdima2\ht\tw@ \divide\@tempdima\thr@@
    \setbox\tw@\hbox{%
       \raise\@tempdima\hbox{\scalebox{1}[-1]{\lower\@tempdima\box
\tw@}}}%
    {\ooalign{\box\tw@ \cr \box\z@}}}

\newcount\foreachcount

\def\foreachletter#1#2#3{\foreachcount=#1
  \ea\loop\ea\ea\ea#3\@alph\foreachcount
  \advance\foreachcount by 1
  \ifnum\foreachcount<#2\repeat}

\def\foreachLetter#1#2#3{\foreachcount=#1
  \ea\loop\ea\ea\ea#3\@Alph\foreachcount
  \advance\foreachcount by 1
  \ifnum\foreachcount<#2\repeat}

\def\definescr#1{\ea\gdef\csname s#1\endcsname{\ensuremath{\mathscr{#1}}\xspace}}
\foreachLetter{1}{27}{\definescr}
\def\definecal#1{\ea\gdef\csname c#1\endcsname{\ensuremath{\mathcal{#1}}\xspace}}
\foreachLetter{1}{27}{\definecal}
\def\definebold#1{\ea\gdef\csname b#1\endcsname{\ensuremath{\mathbf{#1}}\xspace}}
\foreachLetter{1}{27}{\definebold}
\def\definebb#1{\ea\gdef\csname l#1\endcsname{\ensuremath{\mathbb{#1}}\xspace}}
\foreachLetter{1}{27}{\definebb}
\def\definefrak#1{\ea\gdef\csname f#1\endcsname{\ensuremath{\mathfrak{#1}}\xspace}}
\foreachletter{1}{9}{\definefrak} 
\foreachletter{10}{27}{\definefrak}
\def\definebar#1{\ea\gdef\csname #1bar\endcsname{\ensuremath{\overline{#1}}\xspace}}
\foreachLetter{1}{27}{\definebar}
\foreachletter{1}{8}{\definebar} 
\foreachletter{9}{15}{\definebar} 
\foreachletter{16}{27}{\definebar}
\def\definetil#1{\ea\gdef\csname #1til\endcsname{\ensuremath{\widetilde{#1}}\xspace}}
\foreachLetter{1}{27}{\definetil}
\foreachletter{1}{27}{\definetil}
\def\definehat#1{\ea\gdef\csname #1hat\endcsname{\ensuremath{\widehat{#1}}\xspace}}
\foreachLetter{1}{27}{\definehat}
\foreachletter{1}{27}{\definehat}
\def\definechk#1{\ea\gdef\csname #1chk\endcsname{\ensuremath{\widecheck{#1}}\xspace}}
\foreachLetter{1}{27}{\definechk}
\foreachletter{1}{27}{\definechk}
\def\defineul#1{\ea\gdef\csname u#1\endcsname{\ensuremath{\underline{#1}}\xspace}}
\foreachLetter{1}{27}{\defineul}
\foreachletter{1}{27}{\defineul}

\def\autofmt@n#1\autofmt@end{\mathrm{#1}}
\def\autofmt@b#1\autofmt@end{\mathbf{#1}}
\def\autofmt@l#1#2\autofmt@end{\mathbb{#1}\mathsf{#2}}
\def\autofmt@c#1#2\autofmt@end{\mathcal{#1}\mathit{#2}}
\def\autofmt@s#1#2\autofmt@end{\mathscr{#1}\mathit{#2}}
\def\autofmt@f#1\autofmt@end{\mathfrak{#1}}
\def\autofmt@u#1\autofmt@end{\underline{\smash{\mathsf{#1}}}}
\def\autofmt@U#1\autofmt@end{\underline{\underline{\smash{\mathsf{#1}}}}}
\def\autofmt@h#1\autofmt@end{\widehat{#1}}
\def\autofmt@r#1\autofmt@end{\overline{#1}}
\def\autofmt@t#1\autofmt@end{\widetilde{#1}}
\def\autofmt@k#1\autofmt@end{\check{#1}}

\def\auto@drop#1{}
\def\autodef#1{\ea\ea\ea\@autodef\ea\ea\ea#1\ea\auto@drop\string#1\autodef@end}
\def\@autodef#1#2#3\autodef@end{%
  \ea\def\ea#1\ea{\ea\ensuremath\ea{\csname autofmt@#2\endcsname#3\autofmt@end}\xspace}}
\def\autodefs@end{blarg!}
\def\autodefs#1{\@autodefs#1\autodefs@end}
\def\@autodefs#1{\ifx#1\autodefs@end%
  \def\autodefs@next{}%
  \else%
  \def\autodefs@next{\autodef#1\@autodefs}%
  \fi\autodefs@next}

\DeclareSymbolFont{bbold}{U}{bbold}{m}{n}
\DeclareSymbolFontAlphabet{\mathbbb}{bbold}

\mdef\delbar{\overline{\partial}}

\newcommand{\inv}{^{-1}}

\mdef\hf{\textstyle\frac12 }
\mdef\thrd{\textstyle\frac13 }
\mdef\qtr{\textstyle\frac14 }

\SelectTips{cm}{}
\newdir{ >}{{}*!/-10pt/@{>}}    

\mdef\Id{\mathrm{Id}}
\mdef\id{\mathrm{id}}
\alwaysmath{ell}
\alwaysmath{infty}
\alwaysmath{odot}
\mdef\ten{\mathrel{\otimes}}

\mdef\sqten{\mathrel{\boxtimes}}

\mdef\we{\overset{\sim}{\longrightarrow}}
\mdef\leftwe{\overset{\sim}{\longleftarrow}}

\let\maps\colon

\let\xto\xrightarrow

\def\rightarrowtailfill@{\arrowfill@{\Yright\joinrel\relbar}\relbar\rightarrow}
\newcommand\xrightarrowtail[2][]{\ext@arrow 0055{\rightarrowtailfill@}{#1}{#2}}

\def\twoheadrightarrowfill@{\arrowfill@{\relbar\joinrel\relbar}\relbar\twoheadrightarrow}
\newcommand\xtwoheadrightarrow[2][]{\ext@arrow 0055{\twoheadrightarrowfill@}{#1}{#2}}

\def\slashedarrowfill@#1#2#3#4#5{%
  $\m@th\thickmuskip0mu\medmuskip\thickmuskip\thinmuskip\thickmuskip
   \relax#5#1\mkern-7mu%
   \cleaders\hbox{$#5\mkern-2mu#2\mkern-2mu$}\hfill
   \mathclap{#3}\mathclap{#2}%
   \cleaders\hbox{$#5\mkern-2mu#2\mkern-2mu$}\hfill
   \mkern-7mu#4$%
}
\def\rightslashedarrowfill@{%
  \slashedarrowfill@\relbar\relbar\mapstochar\rightarrow}
\newcommand\xslashedrightarrow[2][]{%
  \ext@arrow 0055{\rightslashedarrowfill@}{#1}{#2}}
\mdef\hto{\xslashedrightarrow{}}
\mdef\htoo{\xslashedrightarrow{\quad}}

\long\def\my@drawfill#1#2;{%
\@skipfalse
\fill[#1,draw=none] #2;
\@skiptrue
\draw[#1,fill=none] #2;
}
\newif\if@skip
\newcommand{\skipit}[1]{\if@skip\else#1\fi}
\newcommand{\drawfill}[1][]{\my@drawfill{#1}}

\newif\ifhyperref
\@ifpackageloaded{hyperref}{\hyperreftrue}{\hyperreffalse}
\iftac
  \let\your@state\state
  \def\state#1{\gdef\currthmtype{#1}\your@state{#1}}
  \let\your@staterm\staterm
  \def\staterm#1{\gdef\currthmtype{#1}\your@staterm{#1}}
  \let\defthm\newtheorem
  \def\currthmtype{}
  \ifhyperref
    \def\autoref#1{\ref*{label@name@#1}~\ref{#1}}
  \else
    \def\autoref#1{\ref{label@name@#1}~\ref{#1}}
  \fi
  \AtBeginDocument{%
    \let\old@label\label%
    \def\label#1{%
      {\let\your@currentlabel\@currentlabel%
        \edef\@currentlabel{\currthmtype}%
        \old@label{label@name@#1}}%
      \old@label{#1}}
  }
\else
  \ifhyperref
    \def\defthm#1#2{%
      \newtheorem{#1}{#2}[section]%
      \expandafter\def\csname #1autorefname\endcsname{#2}%
      \expandafter\let\csname c@#1\endcsname\c@thm}
  \else
    \def\defthm#1#2{\newtheorem{#1}[thm]{#2}}
    \ifx\SK@label\undefined\let\SK@label\label\fi
    \let\old@label\label
    \let\your@thm\@thm
    \def\@thm#1#2#3{\gdef\currthmtype{#3}\your@thm{#1}{#2}{#3}}
    \def\currthmtype{}
    \def\label#1{{\let\your@currentlabel\@currentlabel\def\@currentlabel%
        {\currthmtype~\your@currentlabel}%
        \SK@label{#1@}}\old@label{#1}}
    \def\autoref#1{\ref{#1@}}
  \fi
\fi

\newtheorem{thm}{Theorem}[section]

\defthm{cor}{Corollary}
\defthm{prop}{Proposition}
\defthm{lem}{Lemma}
\defthm{sch}{Scholium}
\defthm{assume}{Assumption}
\defthm{claim}{Claim}
\defthm{conj}{Conjecture}
\defthm{hyp}{Hypothesis}
\defthm{fact}{Fact}
\iftac\theoremstyle{plain}\else\theoremstyle{definition}\fi
\defthm{defn}{Definition}
\defthm{notn}{Notation}
\iftac\theoremstyle{plain}\else\theoremstyle{remark}\fi
\defthm{rmk}{Remark}
\defthm{eg}{Example}
\defthm{egs}{Examples}
\defthm{ex}{Exercise}
\defthm{exstar}{* Exercise}
\defthm{exstarstar}{** Exercise}
\defthm{ceg}{Counterexample}

\def\thmqedhere{\expandafter\csname\csname @currenvir\endcsname @qed\endcsname}

\iftac
  \let\c@equation\c@subsection
\else
  \let\c@equation\c@thm
\fi
\numberwithin{equation}{section}

\@ifpackageloaded{mathtools}{\mathtoolsset{showonlyrefs,showmanualtags}}{}

\alwaysmath{alpha}
\alwaysmath{beta}
\alwaysmath{gamma}
\alwaysmath{Gamma}
\alwaysmath{delta}
\alwaysmath{Delta}
\alwaysmath{epsilon}
\mdef\ep{\varepsilon}
\alwaysmath{zeta}
\alwaysmath{eta}
\alwaysmath{theta}
\alwaysmath{Theta}
\alwaysmath{iota}
\alwaysmath{kappa}
\alwaysmath{lambda}
\alwaysmath{Lambda}
\alwaysmath{mu}
\alwaysmath{nu}
\alwaysmath{xi}
\alwaysmath{pi}
\alwaysmath{rho}
\alwaysmath{sigma}
\alwaysmath{Sigma}
\alwaysmath{tau}
\alwaysmath{upsilon}
\alwaysmath{Upsilon}
\alwaysmath{phi}
\alwaysmath{Pi}
\alwaysmath{Phi}
\mdef\ph{\varphi}
\alwaysmath{chi}
\alwaysmath{psi}
\alwaysmath{Psi}
\alwaysmath{omega}
\alwaysmath{Omega}
\let\al\alpha
\let\be\beta

\let\de\delta

\let\om\omega

\let\ze\zeta
\let\th\theta

\let\Om\Omega

\def\xihat{\widehat{\xi}}

\makeatother

\title{The Shape of Infinity}
\author{Michael Shulman}
\mdef\vd{{\vec{d}}}
\mdef\Ud{{\overline{d}}}
\mdef\ud{{\underline{d}}}
\newcommand\frc[2]{{\textstyle\frac{#1}{#2}}}

\setitemize[1]{leftmargin=2em}

\begin{document}
\maketitle

\section{First Introduction: for students}
\label{sec:introduction}

What is there at infinity?

Well, it depends on what there is \emph{not} at infinity.  Consider
first a plane; what is there at the infinity of a plane?

\textbf{Answer \#1:} a point.  In \emph{stereographic
  projection} we place a sphere
tangent to the plane (at the origin, say) and draw lines from the
north pole, mapping each point on the plane to the point where the
corresponding line intersects the sphere.  Now the plane has been
transformed into the sphere minus the North pole, and the North pole
is a `point at infinity'.

This is cool, but it's also weird in some ways.  Since there is just
one point at infinity, if I go off to infinity in one direction I can
then stand at infinity, turn around, and come back from infinity from
any other direction I want to.  Shouldn't, maybe, \emph{this} infinity
be a little different from \emph{that} infinity?

\textbf{Answer \#2:} a circle.  We can add one point in each
direction.  A way to visualize this, similar to stereographic
projection, is \emph{hemispherical central projection}, drawing lines
from the center of the sphere.  Now the plane is transformed into the
Southern hemisphere, and the equator is a circle at infinity.  Here we
have one point for each direction; for instance, parallel lines, or
asymptotic curves, meet off at infinity.

Actually, however, parallel lines now meet in \emph{two} points: one
in each direction.  One might argue that one purpose for adding points
at infinity is to make parallel lines not so different from
intersecting ones (because they should intersect ``at infinity''), but
surely having them intersect in two points makes them rather different
from other lines that intersect only in one point.

Moreover, in this model the points at infinity are weird in another
way: they are different from all other points.  If I go off in one
direction and end up at a point at infinity, I can't keep going any
further in that direction; I have to turn around and come back.  Maybe
this is the way you think of infinity, as a `boundary' of space, but
it's also nice to have all points of the resulting space be the same.

\textbf{Answer \#3:} a `line'.  We do spherical central projection
again, but now we consider a point and its antipodal point to be the
`same' point of the resulting space.  Now the two points at infinity
in opposite directions are actually the same point.  This is called
the \emph{projective plane}, in which parallel lines intersect at a
single point at infinity, making them just like all other pairs of
lines that intersect in a single finite point.  And if I go to
infinity in one direction, I can keep on going and come back from the
opposite direction.

That means, by the way, that all the `lines' in projective space are
actually circles, since they close up at infinity.  Thus, the circle
at infinity can be called a `(projective) line' at infinity, and
usually is.

\medskip Now, what about spaces other than the plane?  We can do similar
things for 3-space or $n$-space, of course.  But what about, say, an
infinite cylinder?
\begin{itemize}
\item We could add a single point at infinity, getting a `pinched torus'.
\item We could add a circle at infinity, getting an ordinary torus.
\item We could add a circle at infinity in a different way to get a
  Klein bottle.
\item We could add one point at infinity in each direction,
  getting---well, actually it's topologically a sphere.
\item We could add a circle at infinity in each direction, getting a
  finite cylinder with boundary.
\end{itemize}

What about a sphere?  How could we add points at infinity to it?  A
torus?  There isn't any place to add points at infinity to these
spaces, because there is no way to ``get off to infinity'' from them.
We say they are already \textbf{compact}, like the projective plane or
the closed hemisphere, and our process of adding points at infinity is
called \textbf{compactification}.

Often in mathematics, when there are many ways to do something, it is
useful to look for the \emph{best} possible way.  Usually if there is
a best possible way, then it contains the most information.  We talk
about these ``best ways'' having a \emph{universal property}.

Here, we can see that different compactifications ``distinguish''
different points at infinity.  The stereographic projection sees only
one point at infinity.  The projective plane sees a whole projective
line at infinity, making distinctions between different ways to go off
to infinity that the stereographic compactification can't tell the
difference between.  The hemispherical compactification makes even
more distinctions: it can tell the difference between going off to
infinity in exactly opposite directions.  We can even see
geometrically the process of ``losing information'' as we pass between
these compactifications.  If we start from a hemisphere and squash all
the equator together into a point, we end up with a sphere.

Thus, it is natural to think that the \emph{best} possible
compactification, in a formal sense, would be one which makes the
\emph{most possible} distinctions between points at infinity.  This
would give us the most possible information about how a space behaves
as we go off to infinity, and would let us recover any other
compactification by squashing some points together.  For instance, can
we do any better than the hemispherical compactification?  Can we
distinguish between, say, going off to infinity as $x\to \infty$ along
the lines $y=0$ and $y=1$?

Our goal in these notes is to construct this ``best possible''
compactification of any space, which is called the
\textbf{Stone-\v{C}ech compactification}.  In order to do this, we'll
need to define ``spaces'' in a very general sense, and make precise
what we mean by ``compact''.

\section{Second introduction: for experts}
\label{sec:expert-intro}

\emph{If you already know what a topological space is, and perhaps
  have seen the Stone-\v{C}ech compactification before, then you may
  be interested in the following comments.  If not, then this section
  probably won't make a whole lot of sense, so I recommend you skip
  ahead to \S\ref{sec:gauge-spaces} on
  page~\pageref{sec:gauge-spaces}.}

\medskip

The goal of these notes is to develop the basic theory of the
Stone-\v{C}ech compactification without reference to open sets, closed
sets, filters, or nets.  In particular, this means we cannot use any
of the usual definitions of topological space.  This may seem like
proposing to run a marathon while hopping on one foot, but I hope to
convince you that it is easier than it may appear, and not devoid of
interest.

These notes began as a course taught at Canada/USA
Mathcamp\footnote{\url{http://www.mathcamp.org}}, a summer program for
mathematically talented high school students.  Mathcampers are generally
very quick and can handle quite abstract concepts, but are
(unsurprisingly) lacking in formal mathematical background.  Thus,
reducing the prerequisite abstraction is important when designing a
Mathcamp class.

It is common in introductory topology courses to use \emph{metric
  spaces} instead of, or at least prior to, abstract topological
spaces.  The idea of ``measuring the distance between points'' is
generally regarded as more intuitive than abstract notions such as
``open set'' or ``closed set''.  In particular, metric spaces can be
used to bridge the gap between $\ep$-$\de$ notions of continuity that
students may have seen in calculus and an abstract definition in terms
of open sets.

The main disadvantage of metric spaces, compared to topological ones,
is that they are not very general.  In particular, Stone-\v{C}ech
compactifications are not usually metrizable, so metric spaces alone
are unsuitable for our purposes here.  However, I was pleased to
discover a notion which I think is not much harder to understand than
a metric space, but which is significantly more general: a set
equipped with a \emph{family} of metrics.\footnote{In these notes I
  depart from the venerable tradition according to which all metrics
  are Hausdorff.  I prefer the philosophy of point-set topology, that
  separation properties should be considered as 
  \emph{properties}, rather than part of a definition.  Thus, my
  ``metrics'' are what are traditionally called ``pseudometrics''.}

In these notes, I have called such a set (plus an inessential closure
condition on its metrics) a \emph{gauge space}.  Some authors call
them \emph{(pseudo)metric uniformities}, due to the amazing theorem (a
version of Urysohn's lemma) that any \emph{uniform space} can be
presented using a family of metrics.\footnote{Uniform spaces,
  sadly often lacking from undergraduate and even graduate curricula,
  are to uniformly continuous functions as topological spaces are to
  continuous functions.  References
  include~\cite{james:top-unif,james:intro-unif,page:uniform,howes:maat}.}
In particular, since all completely regular spaces are uniformizable,
they are also ``gaugeable''.  Thus gauge spaces are a much wider class
than metric spaces: they include in particular all compact
Hausdorff spaces, and thus all Stone-\v{C}ech compactifications.
Pleasingly, complete regularity is also precisely the necessary
condition for a space to embed into its Stone-\v{C}ech
compactification, so this is a very appropriate level of generality
for us.

Gauge spaces are also \emph{better} than topological spaces, in that
they automatically give us more structure than a topology: they give a
uniformity.  Thus we can discuss ``uniform'' concepts such as uniform
continuity and completeness, at a much more general level than that of
metric spaces, but using \ep-\de definitions that are again just the same
as the standard ones for real numbers that students may have seen
before.

(I also believe there is something to be said for gauge spaces over
the more classical ``entourage'' or ``covering'' notions of uniformity
even aside from pedagogy.  In particular, the somewhat obscure and
difficult ``star-refinement'' property of a uniformity is replaced by
the simple operation of dividing $\ep$ by two.)

The presence of uniform structure in a gauge space also provides a
solution to the next problem which arises: how to define compactness.
Compactness is also traditionally a difficult concept for
students.  In a metric space, compactness is equivalent to sequential
compactness, which is easier to understand; but in a gauge space an
analogous statement would have to involve nets or filters, which are
themselves difficult concepts.

Fortunately, we can take an end run around the whole issue, because
compactness of a uniform space is also equivalent to the conjunction
of \emph{total boundedness} and \emph{completeness}.  Interestingly,
these are both uniform rather than topological properties, but their
conjunction is equivalent to the topological property of compactness.
This description of compactness also provides a pleasing way to
construct compactifications (including the Stone-\v{C}ech
compactification): pass to a topologically equivalent totally bounded
gauge, then complete it.

Total boundedness is not so hard to understand as a strong form of
being ``finite in extent'': no matter how small a mesh we want to draw
on our space, we only need finitely many grid points.  But the final
hurdle is defining completeness and completion.  Completeness of
metric spaces is usually defined using Cauchy sequences, while for
uniform spaces one has to use Cauchy nets or filters
instead---something we wanted to avoid!

The solution to this last puzzle comes from Lawvere's
description~\cite{lawvere:metric-spaces} of metric spaces as enriched
categories, and completeness as representability of certain
profunctors (or equivalently, existence of certain weighted limits).
This involves no Cauchy sequences: instead a ``Cauchy point'' is
specified simply by giving what its distances to all actual points
ought to be, plus some very obvious axioms.  It has the further
advantage that no passage to equivalence classes is necessary in
constructing the completion (quotient sets being another traditionally
difficult concept for students).

Clementino, Hofmann, and
Tholen~\cite{ch:top-laxalg,met-top-unif-approach} have shown that
uniform spaces, and their completions~\cite{ch:lawverecplt}, can be
described in a framework which generalizes Lawvere's.  I have not been
able to find such a framework which reproduces gauge spaces exactly,
although they are a special case of the \emph{prometric} spaces
of~\cite{met-top-unif-approach}.  However, a naive generalization of
Lawvere's completeness criterion from metric spaces to gauge spaces
seems to work quite well.

This, then, is how we can define the Stone-\v{C}ech compactification
$\be X$ without ever mentioning open and closed sets, nets, or
filters: work with gauge spaces, define compactness to mean total
boundedness plus completeness, construct completions using
Lawvere-style Cauchy points, and build $\be X$ by completing an
appropriate totally bounded gauge.  The astute reader will notice that
filters make a somewhat disguised appearance in
\S\ref{sec:points-in-betaX}, but overall I believe I can declare
victory in the stated aim of avoiding point-set topological notions.

It turns out that this construction of $\be X$ is also convenient for
proving basic facts about compactness, such as the fact that it is a
topological property (which is not obvious from the definition we are
using!)  I have left this and other interesting facts to the reader as
exercises---of which there are many.  The starred exercises are
generally the more difficult ones (though there is a wide variety of
difficulty within the starred exercises).  But all of the exercises
should be doable using only the concepts and results that have been
introduced in the paper up to that point (perhaps including previous
exercises).

One last unusual concept that appears in these notes is that of
\emph{proximity}.  This is a level of structure that lies in between
uniformity and topology, and can be studied abstractly in its own
right~\cite{nw:proximity}.  In fact, a proximity is essentially the
same as a totally bounded uniformity, and every uniformity has an
``underlying'' proximity which can be thought of as its
``totally-boundedification''.  Proximity is a natural concept to
introduce when discussing compactifications, because totally bounded
uniformities compatible with a given topological space (or
equivalently proximities compatible with such a space) are equivalent
to compactifications thereof.  Proximity is also a natural
midway-point when trying to motivate the notion of ``topology'' or
``continuity'', starting from metric or gauge notions.  However, to
avoid too much proliferation of concepts (and because the notion of
proximity is somewhat weird to think about), I have mostly relegated
proximity notions to the exercises.

\section{Gauge spaces}
\label{sec:gauge-spaces}

A space has a set of points, but it is more than that.  What is it
that makes the plane cohesive, rather than just a collection of
points?  We need to have a way to judge how close together or far
apart points are.  Here's a natural such abstract notion.

\begin{defn}
  A \textbf{metric} on a set $X$ is a function $d\maps X\times X\to
  [0,\infty)$ such that
  \begin{enumerate}
  \item $d(x,x)=0$ for each $x\in X$ (\emph{reflexivity}).
  \item $d(x,y)=d(y,x)$ for each $x,y\in X$ (\emph{symmetry}).
  \item $d(x,y)+d(y,z)\ge d(x,z)$ for each $x,y,z\in X$
    (\emph{transitivity} or the \emph{triangle inequality}).
  \end{enumerate}
  A space equipped with a metric is called a \textbf{metric space}.
\end{defn}

Reflexivity says that a point is as close as possible to itself.
Symmetry says that if $x$ is close to $y$, then $y$ is just as close
to $x$.  Transitivity says that if $x$ is close to $y$ and $y$ is
close to $z$, then $x$ is close to $z$.

\begin{egs}\ 
  \begin{itemize}
  \item \lR\ with $d(x,y) = |x-y|$.  You are probably familiar with
    the triangle inequality in this case.
  \item $\lR^2$ with $d(x,y) = \sqrt{(x_0-y_0)^2+(x_1-y_1)^2}$.  I use
    $x$ and $y$ for points, and natural numbers to label coordinates.
    This is where the triangle inequality gets its name.
  \item More generally, $\lR^n$.
  \item Let $d$ be a metric on $X$ and let $A$ be a subset of $X$.
    Then $d|_{A\times A}$ is a metric on $A$.
  \item This includes $S^2 = \{(x,y,z) \mid x^2+y^2+z^2=1\}\subset
    \lR^3$, for example.
  \item Any set $X$ with $d(x,y)=0$ for all $x,y\in X$.  This is
    called the \emph{indiscrete} metric; all the points are squished
    so close together as to be indistinguishable.
  \item Any set $X$ with \[d(x,y) =
    \begin{cases}
      0 & x=y \\ 1 & x\neq y.
    \end{cases}
    \]
    Here no two distinct points are close together at all.  This is
    called the \emph{discrete} metric.
  \end{itemize}
\end{egs}

We say a metric is \textbf{separated} (or \emph{Hausdorff}) if
$d(x,y)=0$ only when $x=y$.  Many people include this in the
definition of a metric (calling our more general metrics
``pseudometrics'').

A good way to get a feel for a particular metric is to look at
\emph{balls} around points.  Given a point $x\in X$ and a real number
$\ep>0$, we define the \textbf{(open) ball} around $x$ of radius \ep\
to be the set
\[B_d(x,\ep) = \setof{y\in X | d(y,x)<\ep}.
\]
In $\lR^2$, a ball is the interior of a circle centered at $x$ with
radius \ep.  In \lR, it is an open interval centered at $x$ of length
$2\ep$.  In the indiscrete metric, every open ball is the whole space.
In the discrete metric, every open ball is just a single point.  Here
are some other examples:

\begin{itemize}
\item Consider $\lR^2$ with $d'(x,y) = |x_0-y_0| + |x_1-y_1|$.  This
  is also a perfectly good metric, different from the usual one.  It
  is sometimes called the \emph{taxicab} or \emph{Manhattan metric}.
  Now an open ball around $x$ is a diamond-shape.
\item Another metric on $\lR^2$ is $d''(x,y) =
  \max(|x_0-y_0|,|x_1-y_1|)$.  Now an open ball around $x$ is a square
  centered at $x$ with side length $2\ep$.
\end{itemize}

Now a set equipped with a metric is a good notion of `space', but not
quite good enough for our purposes.  Here is one example of why.  We
have seen that $\lR^n$ has a natural metric for all finite $n$, but
what about
\[\lR^\om = \{ (x_0,x_1,x_2,\dots)\mid x_i \in \lR\}?\]
Our natural inclination is to write
\[d(x,y) = \sqrt{(x_0-y_0)^2+(x_1-y_1)^2 + \dots}\]
but of course this makes no sense; what is
\[d\Big((0,0,0,0,\dots), (1,1,1,1,\dots)\Big)?\]
We could try
\[d(x,y) = \max(|x_0-y_0|,|x_1-y_1|,\dots)\]
but then what about
\[d\Big((0,0,0,0\dots), (1,2,3,4\dots)\Big)?\]
And yet, it does seem to make sense to talk about two points in
$\lR^\om$ being close together.  Two points in $\lR^n$ are close
together when all their coordinates are close together, so it makes
sense to say the same in $\lR^\om$.

In fact, $\lR^\om$ can be given sensible metrics.  One possibility is
to choose a function $g:[0,\infty) \to [0,1)$ with the property that
$g(a+b) \le g(a) + g(b)$, and define
\begin{equation}
  d_\infty(x,y) = \sup_i g(|x_i - y_i|).\label{eq:dom1}
\end{equation}
Another somewhat fancier possibility is
\begin{equation}
  d_\Sigma(x,y) = \sum_i g(|x_i - y_i|) \cdot 2^{-i}.\label{eq:dom2}
\end{equation}
One $g$ which works is $g(a) = \min(a,1)$; I'll generally use this one
in the future.

Note that~\eqref{eq:dom1} still makes sense if we replace $\om$ by an
\emph{uncountable} set, but~\eqref{eq:dom2} does not.  On the other
hand, we will see in the next section that~\eqref{eq:dom2} is
``topologically correct'' but not ``uniformly correct'',
whereas~\eqref{eq:dom1} is not even topologically correct.  Therefore,
to treat the uncountable case correctly, we need to generalize the
notion of metric space.  This generalization will also be necessary
for the Stone-\v{C}ech compactification, later on.


If $d,d'$ are two metrics on a set $X$, we say that $d\ge d'$ if for
all $x,y\in X$, $d(x,y)\ge d'(x,y)$.

\begin{defn}
  A \textbf{gauge} on a set $X$ is a nonempty set $\cG = \{d_\alpha\}$
  of metrics on $X$, called \textbf{gauge metrics} or
  \textbf{$X$-metrics}, such that
  \begin{equation}
    \parbox{4in}{for any two gauge metrics $d_1$ and $d_2$, there exists a
      gauge metric $d_3$  with  $d_3\ge d_1$ and $d_3\ge d_2$.}\label{eq:filt}\tag{$\ast$}
  \end{equation}
  A \textbf{gauge space} is a set $X$ equipped with a gauge \cG; we
  write it as $(X,\cG)$ or just $X$ if there is no danger of confusion.
\end{defn}

In particular, a metric space is a gauge space.  However, we also have
other examples.
\begin{itemize}
\item Consider $\lR^2$ with three metrics: $d_0(x,y) = |x_0-y_0|$,
  $d_1(x,y) = |x_1-y_1|$, and $d_{01}(x,y) =
  \max(|x_0-y_0|,|x_1-y_1|)$.  This generalizes easily to a gauge on
  $\lR^n$.
\item We equip $\lR^\om$ with a countably infinite set of metrics
  defined by \[d_N(x,y) = \max_{1\le i\le N}|x_i-y_i|.\]
\end{itemize}

The intuition is that two points in a gauge space are close when the
distance between them is small in \emph{all} the metrics.  This is why
the above gauge on $\lR^\om$ gives what we want.  For example, we call
a gauge space \textbf{separated} (or \emph{Hausdorff}) if whenever
$d(x,y)=0$ for \emph{all} gauge metrics $d$, then $x=y$.

In a gauge space, by an \textbf{open ball} we mean simply an open ball
with respect to \emph{some} metric in the gauge.  We write
$B_d(x,\ep)$ to specify the metric $d$ being used.

We call condition~\eqref{eq:filt} the \emph{filteredness property}.
We can omit it (and some people do), but in that case there is a
``whack-a-mole'' behavior that comes into play: in some other places
the definitions become more complicated.  We'll come back to it later;
for now we note that given two open balls $B_{d_1}(x,\ep_1)$ and
$B_{d_2}(x,\ep_2)$, the filteredness condition ensures that we can
find a metric $d_3$ such that $B_{d_3}(x, \min (\ep_1,\ep_2))$ is
contained in both $B_{d_1}(x,\ep_1)$ and $B_{d_2}(x,\ep_2)$.

Note that any set \cM of metrics \emph{generates} a gauge by adding to
it the metrics $\max (d_0,\dots, d_n)$ for any finite set of metrics
$d_0,\dots, d_n \in \cM$.  But we can't we allow infinite subsets of
\cM, because then the maximum might not exist---even a supremum might
not exist!  We already had this problem with $\lR^\om$.  We could try
to circumvent it as in~\eqref{eq:dom1}, but we'll see in the next
section that this would be the wrong thing to do.

\section{Topology}
\label{sec:topology}

Now, when we did stereographic projection, we identified the plane
with the points on the sphere that their lines meet.  But that
bijection does \emph{not} preserve distance!  What does it preserve?
A word for it is \emph{topology}, but what does that mean?  We'd like
to say that it preserves ``infinite closeness,'' but no two points can
be infinitely close without being at distance zero (which, in most
spaces, is only possible if they are the same point).  However,
\emph{sets} can be infinitely close without being identical (or even
intersecting).



\begin{defn}
  For $d$ a metric on $X$ and nonempty subsets $A,B\subseteq X$ we define
  \[ d(A,B) = \inf_{a\in A, b\in B} d(a,b) \]
  For $x\in X$ we write
  \[ d(x,B) = d(\{x\},B) \]
  In a gauge space, we write $A\approx B$ if $d(A,B)=0$ for all gauge
  metrics $d$.  By convention, we say $\emptyset \not\approx A$ for all
  sets $A$.
\end{defn}

Evidently if $A\cap B \neq \emptyset$, then $A\approx B$.  Can you think
of an example of two sets with $A\cap B = \emptyset$ and $A\approx B$?

\begin{egs}
  \begin{itemize}
  \item Let $A$ be the open disc in the plane with center $(-1,0)$ and
    radius $1$, and $B$ the similar disc with center $(1,0)$.
  \item Let $A$ be the curve $y=1/x$ in the plane, and $B$ the curve $y=-1/x$.
  \item Let $A$ be the set $\setof{1,\frac12,\frac13, \frac14,\dots}$
    and $B$ the set $\setof{0}$.
  \end{itemize}
\end{egs}

\begin{lem}
  $A\approx B$ iff for every gauge metric $d$ and every $\ep>0$, there
  exists $a\in A$ and $b\in B$ with $d(a,b)< \ep$.
\end{lem}
\begin{proof}
  Easy.
\end{proof}

Does the stereographic bijection preserve $\approx$?  No; two parallel
lines in $\lR^2$ have $d(A,B)$ positive (their distance apart) but
their stereographic projections are at distance 0.  But it \emph{does}
preserve $\approx$ between a \emph{point} and a set (though this may
not be obvious yet).  This leads us to separate the following two notions.

\begin{itemize}
\item The \textbf{proximity} of a gauge space is the relation $A\approx
  B$ on pairs of subsets.
\item The \textbf{topology} of a gauge space is the relation $x\approx
  A$ on (point, subset) pairs.
  (We write $x\approx A$ to mean $\setof{x}\approx A$.)
\end{itemize}
When we say that a concept or definition is \emph{topological}, we
mean that it depends only on the relation ``$x\approx A$'' between
points and subsets.  Similarly, we can talk about a concept or
definition being \emph{proximal}.

Two different gauges on a set can define the same topology or
proximity.  In the case of topologies, there is a nice criterion for
this.

\begin{lem}
  Two gauges \cG and $\cG'$ on a set $X$ define the same topology (are
  \emph{topologically equivalent}) if and only if both
  \begin{itemize}
  \item For every metric $d\in\cG$ and open ball $B_d(x,\ep)$, there
    exists a metric $d'\in\cG'$ and an open ball
    $B_{d'}(x,\ep')\subseteq B_d(x,\ep)$, and
  \item For every metric $d'\in\cG'$ and open ball $B_{d'}(x,\ep')$,
    there exists a metric $d\in\cG$ and an open ball
    $B_{d}(x,\ep)\subseteq B_{d'}(x,\ep')$.
  \end{itemize}
\end{lem}
\begin{proof}
  Suppose they define the same topology, and given $B_d(x,\ep)$, and
  suppose contrarily that every open ball $B_{d'}(x,\ep')$ contains a
  point outside $B_d(x,\ep)$.  Then for every $d'\in\cG'$ and
  $\ep'>0$, there exists a point $y$ in $X \setminus B_d(x,\ep)$ with
  $d'(x,y)<\ep'$, so $x \approx (X\setminus B_d(x,\ep))$ relative to
  $\cG'$.  But $\cG$ and $\cG'$ define the same topology, so $x
  \approx (X\setminus B_d(x,\ep))$ also relative to $\cG$.  In
  particular, there exists a point $y$ outside of $B(x,\ep)$ such that
  $d(x,y)<\ep$, clearly absurd.

  Conversely, suppose the ball conditions and that $x\approx A$ with
  respect to $\cG$, i.e.\ for every $d\in\cG$ and $\ep>0$ there is a
  $y\in A$ with $d(x,y)<\ep$.  Suppose given a $d'\in\cG'$ and an
  $\ep'>0$; then there is a $B_d(x,\ep) \subseteq B_{d'}(x,\ep')$,
  hence a $y\in A \cap B_d(x,\ep) \subseteq Y \cap B_{d'}(x,\ep')$,
  which is what we want.
\end{proof}

The relation $\approx$ is a good way to \emph{think} about topological
equivalence, but the conditions in this lemma are usually the best way
to \emph{work} with it.

\begin{eg}
  The usual metric, the taxicab metric, and the maximum metric on
  $\lR^n$ are all topologically equivalent.  They are also equivalent
  to the gauge generated by the metrics $d_i(x,y) = |x_i - y_i|$.
\end{eg}

Similarly, we say that a bijection $f:X\leftrightarrow Y$ is a
\textbf{topological isomorphism} or a \textbf{homeomorphism} if we have
$x\approx A \iff f(x) \approx f(A)$.  There is an analogous condition
for this.  Some examples:

\begin{itemize}
\item Stereographic projection is a topological isomorphism from the
  plane $\lR^2$ to $S^2 \setminus \{N\}$.
\item The real line $\lR$ is topologically isomorphic to the open
  interval $(0,1)$, via $f(x) = \frac12+\frac1\pi\arctan x$ whose
  inverse is $f\inv(x) = \tan (\pi(x-\frac12))$.
\item In the discrete metric, replacing 1 with any other
  positive real number produces a different, but topologically
  equivalent, metric.
\end{itemize}

So the first step in the process of compactification involves passing
to a new topologically isomorphic space, in which it is somehow
``easier to see'' where to put points at infinity.  Soon we'll think
about what property of the new space that is, and how to add the
missing points.

\section{Exercises on gauges}
\label{sec:ex-gauge}

\begin{ex}
  Does the distance between subsets
  \[ d(A,B) = \inf_{a\in A, b\in B} d(a,b)
  \]
  define a metric on the power set $\cP X$ of $X$?
\end{ex}

\begin{ex}
  There is a gauge on any set $X$ which consists of \emph{all}
  possible metrics on $X$.  Prove that this gauge is topologically
  equivalent to a discrete metric.
\end{ex}


\begin{ex}\label{ex:fin-gauge}
  Prove that any gauge containing only a finite number of metrics is
  topologically equivalent to some gauge with only a single metric.
\end{ex}

\begin{ex}
  Prove that there is a largest gauge that is topologically equivalent
  to any given gauge.  That is, prove that for any gauge \cG\ there is
  a gauge $\cH$, which is equivalent to \cG, and such that if $\cG'$
  is topologically equivalent to \cG, then $\cG'\subseteq\cH$.
\end{ex}

\begin{ex}\label{ex:eqv-bdd}
  A metric is \textbf{bounded} if there is an $N\in[0,\infty)$ such
  that $d(x,y)<N$ for all points $x,y$.  Prove that every gauge is
  topologically equivalent to a gauge containing only bounded metrics.
\end{ex}

\begin{ex}
  Consider the following two metrics on $\lR^\om$:
  \begin{align*}
    d_\infty(x,y) &= \sup_i \big( \min(|x_i - y_i|,1)\big)\\
    d_\Sigma(x,y) &= \sum_{i=0}^\infty \min(|x_i-y_i|,1) \cdot 2^{-i}
  \end{align*}
  Convince yourself that both are, indeed, metrics.  Is either of them
  (that is, the singleton gauges $\{d_\infty\}$ and $\{d_\Sigma\}$)
  topologically equivalent to the gauge on $\lR^\om$ defined above,
  which consisted of metrics
  \[d_N(x,y) = \max_{1\le i\le N}|x_i-y_i| \qquad ?\]
\end{ex}

\begin{exstar}
  Find a set $X$ admitting a gauge which you can \emph{prove} is
  not topologically equivalent to any single metric.
\end{exstar}

\begin{exstar}\label{ex:top-space}
  A \textbf{topology} on a set $X$ is usually defined to be a
  collection of subsets of $X$, called \emph{open sets}, which is
  closed under finite intersections and arbitrary unions.  (This
  includes the intersection of no sets, which is $X$, and the union of
  no sets, which is $\emptyset$.)  Prove that every gauge on $X$ gives
  rise to a topology on $X$, and that this topology is exactly
  characterized by the relation ``$x\approx A$'' between points and
  subsets.
\end{exstar}

\begin{exstar}\label{ex:proximity}
  Suppose \cG and $\cG'$ are two gauges on the same set $X$.
  \begin{enumerate}
  \item Find an \ep-style condition which is sufficient to ensure that
    \cG and $\cG'$ define the same \emph{proximity}.
  \item Find an \ep-style condition which is both sufficient and
    \emph{necessary} to ensure that \cG and $\cG'$ define the same
    proximity.
  \item Which of the previous topological equivalences are actually
    proximal equivalences?
  \end{enumerate}
\end{exstar}

\begin{ex}\label{ex:quasi-gauge}
  A \textbf{quasi-metric} on a set $X$ is a function $d\maps X\times X
  \to \lR$ satisfying the reflexivity and transitivity properties of
  a metric, but not necessarily symmetry.  A \textbf{quasi-gauge} is a
  set of quasi-metrics satisfying the same ``filteredness'' condition as a
  gauge.  See how much of the theory of gauge spaces you can mimic for
  quasi-gauges (you may want to come back to this exercise as we
  develop more of this theory).
\end{ex}

\begin{exstar}
  Prove that every topology on a set $X$ (see \autoref{ex:top-space})
  is induced by some quasi-gauge on $X$ (see \autoref{ex:quasi-gauge}).
\end{exstar}

\section{Continuity}
\label{sec:continuity}


The most important kind of morphism between gauge spaces is one that
preserves the topology (only).

\begin{defn}
  A function $f:X\to Y$ between gauge spaces is \textbf{continuous} if
  whenever $x\approx A$ in $X$, then $f(x) \approx f(A)$ in $Y$.
\end{defn}

As we did for topological equivalence, we can reformulate this in a less
intuitive, but more useful, way.

\begin{lem}
  $f\maps X\to Y$ is continuous iff for any $x\in X$, any $Y$-metric
  $d_Y$, and any $\ep>0$, there exists an $X$-metric $d_X$ and a
  $\de>0$ such that for any $x'\in X$, if $d_X(x,x')<\de$ then
  $d_Y(f(x),f(x'))<\ep$.
\end{lem}
\begin{proof}
  Suppose $f$ is continuous, and given $x$, $d_Y$, and $\ep$, and
  suppose the contrary: for every $d_X$ and $\de$ there exists an
  $x'\in X$ with $d_X(x,x')<\de$ but $d_Y(f(x),f(x')) \ge \ep$.
  Define
  \[ A = \setof{ x'\in X | d_Y(f(x),f(x')) \ge \ep } \]
  Then the assumption implies that $x\approx A$.  Hence $f(x) \approx
  f(A)$, so there should be a point $y \in f(A)$ with
  $d_Y(f(x),y)<\ep$, which is absurd.

  Conversely, suppose the condition and that $x\approx A$.  To
  prove $f(x) \approx f(A)$, say given $d_Y$ and $\ep>0$; we want to
  find a $y\in f(A)$ with $d_Y(f(x),y)<\ep$.  Choose $d_X$ and $\de$
  as in the condition; then since $x\approx A$ there exists $x'\in A$
  with $d_X(x,x')<\de$.  Take $y = f(x')$.
\end{proof}

Examples:
\begin{itemize}
\item The identity function $\id_x\maps X\to X$ is always continuous.
\item Any constant function $f\maps X\to Y$, $f(x)=a$ for some fixed
  $a\in Y$, is continuous.
\item Addition $+\maps \lR^2\to\lR$ is continuous; take $\de =
  \ep/2$.
\item Multiplication $\cdot\maps \lR^2\to\lR$ is continuous.  Given
  $(x_0,x_1)\in\lR^2$, take
  \[\de = \min\left(\frac{\ep}{3|x_0|}, \frac\ep {3|x_1|}, \sqrt{\frac{\ep}3}\right).\]
\item The composite of continuous functions is continuous.
\item Therefore, any polynomial $p\maps \lR\to\lR$ is continuous
  (using the diagonal as well).
\item If $X$ has a discrete gauge or $Y$ has an indiscrete gauge, then
  any function $X\to Y$ is continuous.
\end{itemize}

Non-examples:
\begin{itemize}
\item The step function \[\th(x) =
  \begin{cases}
    0 & x \le 0\\ 1 & x > 0
  \end{cases}\]
  is not continuous.  Take $x=0$ and $\ep < 1$.  Then there is no ball
  around $0$ such that everything in that ball is mapped to within
  \ep\ of $\th(0)=0$.
\item Similarly for the blip function \[\de(x) =
  \begin{cases}
    0 & x\neq 0\\ 1 & x = 0
  \end{cases}.\]
\end{itemize}

A continuous function preserves all the topological properties of a
space.  It extends our previous notion of topological isomorphism, as
follows.

\begin{thm}
  A bijection $f:X\leftrightarrow Y$ is a topological isomorphism if
  and only if both it and its inverse are continuous.
\end{thm}
\begin{proof}
  Essentially by definition.
\end{proof}

We saw lots of topological isomorphisms last time.  Here is an
important sort of counterexample to keep in mind.

\begin{eg}
  The map $t\mapsto (\cos t, \sin t)$ from $[0,2\pi)$ to $S^1$ is
  continuous and bijective, but not a topological isomorphism.
\end{eg}

The notion of continuous function gives a reason why we don't want to
use $d_\infty(x,y) = \max_{i\in\lN} |x_i-y_i|$ as a single metric on
$\lR^\om$, and similarly why we don't want to close up a set of
metrics under infinite suprema.

\begin{prop}
  Suppose that $\lR^\om$ has the metric
  \[ d_\infty(x,y) = \sup_i \big( \min(|x_i - y_i|,1)\big)\] and \lR\ the
  usual one.  Then the function $f(t) =
  (t,\sqrt{|t|},\sqrt[3]{|t|},\sqrt[4]{|t|},\dots)$ from \lR\ to
  $\lR^\om$ is not continuous at $t=0$.
\end{prop}
\begin{proof}
  Fix $\ep>0$.  Then $|\sqrt[n]{|t|}|<\ep$ is equivalent to $|t| <
  \ep^n$.  Since for any $t\neq 0$, there is an $n$ such that $|t| >
  \ep^n$, there is no $\de>0$ such that $|t|<\de$ implies
  $|\sqrt[n]{|t|}|<\ep$ for all $n$.
\end{proof}

However, the actual gauge on $\lR^\om$ we used does work.  More
generally, we have the following.  Let $X$ be a gauge space, $A$ a
set, and $X^A$ the set of all functions $A\to X$, which we write as
$(x_a)_{a\in A}$.  Let $\pi_a \maps X^A \to X$ be defined by $\pi_a(x)
=x_a$ for each $a\in A$.  For each $X$-metric $d$ and
each \emph{finite} subset $B\subseteq A$, define a metric $d_B$ on $X^A$
by $d_B(x,y) = \max_{a\in B} d(x_a,y_a)$.  The set of all these
metrics, as $B$ varies over finite subsets of $A$, is called the
\emph{pointwise gauge} on $X^A$.

\begin{prop}\label{thm:pointwise}
  For any gauge space $Z$, a function $f\maps Z\to X^A$ is continuous
  (where $X^A$ has the pointwise gauge) if and only if for each $a\in
  A$, the function $\pi_a \circ f\maps Z\to X$ is continuous.
\end{prop}

In other words, a function is continuous if and only if all its
coordinates are continuous.

\begin{proof}
  Firstly, the function $\pi_a \maps X^A \to X$ defined by $\pi_a(x)
  =x_a$ is evidently continuous; the preimage of each $B_d(x,\ep)$
  contains $B_{d_{\{a\}}}(x,\ep)$.  Thus, $\pi_a\circ f$ is continuous
  if $f$ is.  (This much is also true for the metric $d_\infty$.)

  Now suppose that each $\pi_a \circ f$ is continuous.  Thus, for each
  $a\in A$, $X$-metric $d$, and ball $B_d(f(z)_a,\ep)$ there is a
  $Z$-metric $d_a$ and a ball $B_{d_a}(z,\de_a)$ contained in the
  preimage of $B_d(f(z)_a,\ep)$.  Choose any ball around $f(z)$ in
  $X^A$; it is of the form $B_{d_B}(f(z),\ep)$ for some finite
  $B\subseteq A$.  By filteredness of the gauge on $Z$, there is a
  $Z$-metric $d'$ with $d'\ge d_a$ for all $a\in B$.  If we let $\de =
  \min_{a\in B} \de_a$, then it follows that $B_{d'}(z,\de)$ is
  contained in the preimage of $B_{d_B}(f(z),\ep)$, as desired.
\end{proof}

Here is the first place we've used filteredness, but this may not be
satisfying: we only used it because we decided to close up the metrics
on $X^A$ under finite maxes.  What if we didn't close them up under
any maxes at all?

Consider $\lR\times\lR$ with the two metrics $d_0(x,y) = |x_0-y_0|$
and $d_1(x,y) = |x_1-y_1|$.  This is not a gauge by our definition,
but it would be if we omitted the filteredness condition.  But in this
``gauge'', and with our definition of $\approx$, we would have
\[ (0,0) \approx \{(1,0), (0,1) \} \] because for each of $d_0$ and
$d_1$, there is an element of the right-hand set which is at distance
zero from the left-hand point.  Clearly this is not the right topology
on $\lR\times \lR$.  Thus, if we didn't require the filteredness
condition, we would have to define $A\approx B$ with reference to
arbitrary finite sets of metrics, which is more of a pain to think
about.  (We will, however, have to do something similar later on when
we talk about evaluation data in \S\ref{sec:points-in-betaX}.)

\begin{rmk}
  Lest the reader come away with too negative an impression of
  $d_\infty$, let me say that it also has important uses.  It is
  called the \textbf{supremum metric} or \textbf{sup-metric}, and
  induces the \textbf{topology of uniform convergence} on $X^A$.
\end{rmk}

\section{Exercises on continuity}
\label{sec:ex-continuity}

\begin{ex}
  Let $X$ and $Y$ be metric spaces (that is, gauge spaces whose gauge
  contains only one metric).  A \textbf{contraction} is a function
  $f\maps X\to Y$ such that $d(f(x),f(x')) \le d(x,x')$.  Prove that
  any contraction is continuous.
\end{ex}

\begin{ex}\label{ex:dist-cts}
  Let $A\subseteq X$ be nonempty.  Prove that for any $X$-metric $d$,
  the function $d_A\maps X\to \lR$ defined by $d_A(y) = d(A,y)$ is
  continuous.
\end{ex}

\begin{ex}
  Prove that two gauges on the same set $X$ are topologically
  equivalent if and only if they agree about which functions $f:X\to
  \lR$ are continuous.
\end{ex}



\begin{ex}\label{ex:sep-quot}
  We can consider the relation $\approx$ between pairs of
  \emph{points} of a gauge space as well: we have $x\approx y$ iff
  $d(x,y)=0$ for all gauge metrics $d$.  Prove that:
  \begin{enumerate}
  \item $\approx$ is an equivalence relation on $X$.
  \item The quotient $X/\approx$ is a separated gauge space.
  \item The quotient map $X\to X/\approx$ is continuous.
  \end{enumerate}
\end{ex}

\begin{ex}\label{ex:collapsing-metric}
  Let $(X,\cG)$ be a gauge space and $A\subseteq X$ a subset.  For each
  gauge metric $d$, define a new metric by
  \[d_A(x,y) = \min\big(d(x,y), d(x,A)+d(A,y)\big).\]
  \begin{enumerate}
  \item Prove that $d_A$ is a metric on $X$, and that the set of these
    metrics defines a new gauge $\cG_A$ on $X$.
  \item Prove that the identity is a continuous function $(X,\cG)\to
    (X,\cG_A)$.
  \end{enumerate}
\end{ex}

\begin{exstar}\label{ex:top-cts}
  If $X$ and $Y$ are sets equipped with topologies as in
  \autoref{ex:top-space}, a function $f:X\to Y$ is usually defined to
  be \emph{continuous} if for every open subset $U\subseteq Y$, the
  preimage $f^{-1}(U)\subseteq X$ is also open.  Prove that a function
  between gauge spaces is continuous in our sense if and only if it is
  continuous in this sense with respect to the underlying topologies
  you defined in \autoref{ex:top-space}.
\end{exstar}


\begin{exstar}\label{ex:proxcts}
  A function $f:X\to Y$ is called \textbf{proximally continuous} if
  $A\approx B$ in $X$ implies $f(A) \approx f(B)$ in $Y$.  Find an
  \ep-style condition which is sufficient (or, better, necessary and
  sufficient) for $f$ to be proximally continuous.  Which of the
  continuous functions we have looked at are proximally continuous?
  (See \autoref{ex:proximity}.)
\end{exstar}

\section{Total boundedness}
\label{sec:tbdd}

Our examples of compactification of $\lR^2$ all had two steps.  First
we mapped $\lR^2$ to a space that was `finite in extent', making the
desired `points at infinity' into `holes' at a finite location.
Second, we then added those points to make the space compact.  Our
goal is to formalize both of those steps to apply them to arbitrary
gauge spaces.

Here is a rough roadmap of the concepts that we will introduce.

\begin{itemize}
\item A space is called \textbf{compact} if you can't escape from it,
  either to infinity or into a hole.
\item A space is called \textbf{totally bounded} if it is `finite in
  extent', or equivalently if you can't escape to infinity (although
  you might escape into a hole).
\item A space is called \textbf{complete} if it has no holes, so the
  only way to escape from it is to infinity.
\end{itemize}
Clearly a space should be compact if and only if it is both totally
bounded and complete.  Some examples:
\begin{itemize}
\item Sphere, toruses, discs, $RP^n$ are all compact.
\item A sphere minus a point is totally bounded, but not complete,
  since it has a hole.
\item An ordinary plane is complete, but not totally bounded, since
  it is `infinite in extent'.
\item A plane minus a point is neither complete nor totally bounded.
\end{itemize}

It turns out that compactness is a topological property, i.e.\ it is
invariant under topological equivalence.  (This will not be obvious
from our definition, but it is true.)  But total boundedness is
\emph{not} a topological property, and neither will completeness be.
This is a \emph{good} thing for our process of compactification:
\begin{enumerate}[label=\arabic*.]
\item Given a gauge space, find a topologically equivalent gauge which
  is totally bounded.
\item `Complete' the resulting totally bounded gauge space, to produce
  a space which is both totally bounded and complete, hence compact.
\end{enumerate}
It is possible to compactify `in one step' by directly adding points
at infinity.  However, it turns out to be easier, and more
comprehensible, to first bring those points `in from infinity' and
then add them at some finite location.

Now let's try to define ``total boundedness'' precisely.  What
important properties does, say, the sphere-minus-a-point have which
distinguish it from the plane for our purposes?

Intuitively, it is ``bounded'' in some sense.  The obvious notion of
\emph{boundedness} for a metric space $X$ is that $d(x,y)<N$ for some
$N$.  Equivalently, if $X = B(x,N)$ for some $x$ and $N$.  However,
this is not a very well-behaved concept.  Suppose $d$ is any metric;
then we can define a new metric by
\[ d'(x,y) = \min(d(x,y),1) \]
This is bounded, but at small distance scales it looks like $d$.  For
instance, on $\lR^2$ we have
\[ d'(x,y) = \min\left( \sqrt{(x_0-y_0)^2 + (x_1-y_1)^2}, 1 \right) \]
Then we have $\lR^2 = B_{d'}(\mathbf{0},N)$ for any $N> 1$.  But if we
shrink $N$ just a little, this becomes \emph{wildly false}: it takes
infinitely many $d'$-balls of radius $1$ to cover $\lR^2$.

This doesn't happen on a sphere: for any $\ep>0$, it only takes
finitely many balls of radius $\ep$ to cover the sphere.  Thus, the
following definition captures this notion of being ``bounded at
arbitrarily small distance scales.''

\begin{defn}
  A gauge space $X$ is \textbf{totally bounded} if for any $X$-metric
  $d$ and $\ep>0$, we can write $X$ as the union of finitely many
  balls $B_d(x_i,\ep)$:
  \[ X = B_d(x_0,\ep) \cup \cdots\cup B_d(x_n,\ep). \]
\end{defn}

Of course, $X$ can always be covered by \emph{some} number of
$(d,\ep)$-balls.  Take, for instance, \emph{all} the balls
$B_d(x,\ep)$ for $x\in X$.  The point is whether we can do it with
finitely many.

Equivalently, $X$ is totally bounded if for any $d$ and $\ep>0$, there
exists a finite subset $A\subseteq X$ such that $d(x,A)<\ep$ for all
$x\in X$.

\begin{egs}
  \begin{itemize}
  \item The unit square $[0,1]^2$ is totally bounded; given any \ep\
    we can take a finite \ep-spaced grid as the centers of \ep-sized
    balls.
  \item The plane $\lR^2$ is not totally bounded; any finite number of
    balls has only finite area.
  \item Any subspace of a totally bounded space is totally bounded.
    Thus, for instance, $\lQ \cap [0,1]$ is totally bounded, even
    though it has lots and lots of holes.
  \end{itemize}
\end{egs}

\section{Uniformity}
\label{sec:uniformity}

Note that total boundedness is \emph{not} a topological property: the
plane is not totally bounded, but the sphere-minus-a-point is, and
they are topologically isomorphic.  In fact, total boundedness is not
even a proximity concept!

\begin{eg}
  Let $X$ be an arbitrary set, and consider the following two gauges
  on $X$.  Let \cG be the discrete gauge, containing only one metric
  with $d(x,y) = 1$ iff $x\neq y$.

  To define $\cG'$, suppose we have a finite partition of $X$:
  \[ X = A_1 \sqcup \dots \sqcup A_n.
  \]
  (Writing $\sqcup$ instead of $\cup$ means that $A_i \cap A_j =
  \emptyset$ for $i\neq j$.)  Then define a metric on $X$ by
  \[ d_{A_1,\dots,A_n}(x,y) =
  \begin{cases}
    0 &\text{if } x,y\in A_i \text{ for some } i\\
    1 &\text{if } x\in A_i,\; y\in A_j \text{ for } i\neq j.
  \end{cases}
  \]
  The collection of these metrics, for all partitions of $X$, defines
  a gauge $\cG'$ on $X$.

  I claim that these two gauges are proximally equivalent.  First of
  all, in the discrete gauge, we clearly have $A\approx B$ iff $A\cap
  B\neq\emptyset$.  Since $A\cap B\neq\emptyset$ always implies
  $A\approx B$, it suffices to prove that if $A\cap B =\emptyset$,
  then $A\not\approx B$ in $\cG'$.  But defining $C = X \setminus
  (A\cup B)$, we have a partition $X = A \sqcup B \sqcup C$, and
  clearly
  \[ d_{A,B,C}(A,B) = 1\]
  so that $A\not\approx B$.

  However, $\cG'$ is totally bounded, but (if $X$ is infinite) \cG is
  not.  The latter is obvious: there is no finite cover of $X$ by
  \ep-sized balls for any $\ep<1$.  The former is almost as obvious:
  for any partition and $\ep>0$, each $A_i$ is contained in some
  $(d_{A_1,\dots,A_n},\ep)$-ball (in fact, if $\ep<1$ it is equal to
  some such ball), so that $n$ such balls suffice.
\end{eg}

There ought to be some notion of equivalence which does preserve total
boundedness, though.  It turns out that the following is the right
one.

\begin{defn}
  A function $f:X\to Y$ between gauge spaces is \textbf{uniformly
    continuous} iff for any $Y$-metric $d_Y$ and $\ep>0$, there exists
  an $X$-metric $d_X$ and $\de>0$ such that \emph{for all $x,x'\in
    X$}, we have
  \[ d_X(x,x') <\de \;\Longrightarrow\; d_Y(f(x),f(x'))<\ep. \]
\end{defn}

The difference between uniform continuity and continuity is that in
the uniform case, $\de$ is only allowed to depend on $\ep$, not on
$x$.

\begin{eg}
  $f:\lR\to\lR$ defined by $f(x)=x^2$ is not uniformly continuous.
\end{eg}

Thus, uniform continuity is a very strong concept!

\begin{defn}
  A \textbf{uniform isomorphism} is a bijection $X\leftrightarrow Y$
  which is uniformly continuous in both directions.
\end{defn}

\begin{prop}
  If $f:X\to Y$ is a uniform surjection and $X$ is totally
  bounded, so is $Y$.
\end{prop}
\begin{proof}
  Suppose given a $Y$-metric $d_Y$ and $\ep>0$, and choose $d_X$ and
  $\de>0$ as in uniform continuity.  Since $X$ is totally bounded, it
  is the union of finitely many $(d_X,\de)$-balls.  By definition, the
  $f$-image of each of those balls is contained in a $(d_Y,\ep)$-ball,
  and $f$ is surjective, so finitely many of the latter suffice to
  cover $Y$.
\end{proof}

\begin{cor}
  If $X\leftrightarrow Y$ is a uniform isomorphism and $X$ is totally
  bounded, so is $Y$.
\end{cor}

\section{Exercises on total boundedness and uniformity}
\label{sec:ex-tbdd-unif}

\begin{ex}
  Which of the three topologically metrics on $\lR^2$ that we
  considered in \S\ref{sec:gauge-spaces} are uniformly equivalent?
\end{ex}

\begin{ex}\label{ex:dist-ucts}
  Prove that the functions $d_A$ from \autoref{ex:dist-cts} are, in
  fact, uniformly continuous.
\end{ex}

\begin{ex}
  Show that the metric $d_\Sigma$ on $\lR^\om$ is not uniformly
  equivalent to the many-metric gauge defined in
  \S\ref{sec:gauge-spaces}.
\end{ex}

\begin{ex}\label{ex:ucts-pcts}
  Prove that any uniformly continuous function is proximally
  continuous.  Conclude that any two uniformly equivalent gauges are
  also proximally equivalent.
\end{ex}

\begin{ex}
  Which of the other continuous functions (or topological
  isomorphisms) that we have seen so far are actually uniformly
  continuous (or uniform isomorphisms)?
\end{ex}

A sequence $(x_0,x_1,x_2,\dots)$ in a gauge space $X$ is said to
\textbf{converge} to $x_\infty\in X$ if for any gauge metric $d$ and
any $\ep>0$, there exists an $N\in\lN$ such that $n>N \Longrightarrow
d(x_n,x_\infty)<\ep$.

\begin{ex}
  Prove that if $(x_0,x_1,x_2,\dots)$ converges to $x_\infty$ in $X$, and
  $f:X\to Y$ is continuous, then $(f(x_0),f(x_1),f(x_2),\dots)$
  converges to $f(x_\infty)$ in $Y$.
\end{ex}

A sequence $(x_0,x_1,x_2,\dots)$ is said to be \textbf{Cauchy} if
for any gauge metric $d$ and any $\ep>0$, there exists an $N\in\lN$
such that if $n,m>N$, then $d(x_n,x_m)<\ep$.

\begin{ex}
  Prove that if a sequence in a gauge space converges, then it is
  Cauchy.
\end{ex}

\begin{ex}
  Prove that if $(x_0,x_1,x_2,\dots)$ is Cauchy in $X$, and $f:X\to Y$
  is \emph{uniformly} continuous, then $(f(x_0),f(x_1),f(x_2),\dots)$
  is Cauchy in $Y$.
\end{ex}

\begin{ex}
  Prove that $X$ is a gauge space such that every sequence in $X$ has
  a Cauchy subsequence, then $X$ is totally bounded.
\end{ex}

\begin{ex}
  Prove that if $X$ is a totally bounded \emph{metric} space (a gauge
  space with exactly one gauge metric), then every sequence in $X$ has
  a Cauchy subsequence.
\end{ex}

\begin{exstar}
  Find an example of a totally bounded gauge space which contains a
  sequence that has no Cauchy subsequence.
\end{exstar}

\begin{exstar}\label{ex:tbddproxunif}
  Prove that if $f:X\to Y$ is proximally continuous and $Y$ is totally
  bounded, then $f$ is uniformly continuous.  Conclude that if two
  totally bounded gauges on a set $X$ are proximally equivalent, then
  they are in fact uniformly equivalent.
\end{exstar}

\begin{exstar}
  Prove that if $f:X\to Y$ is proximally continuous and $X$ is a
  metric space, then $f$ is uniformly continuous.
\end{exstar}

\begin{exstar}\label{ex:pscpt}
  A gauge space $X$ is \textbf{pseudocompact} if every continuous
  $f:X\to \lR$ is bounded.
  \begin{enumerate}
  \item Prove that if $X$ is pseudocompact, then every continuous
    $f:X\to \lR$ achieves a maximum and a minimum.\label{item:psc1}
  \item Prove that any pseudocompact gauge space is totally
    bounded.\label{item:psc2}
  \end{enumerate}
\end{exstar}

\section{Completeness and completion}
\label{sec:cauchy-points}

Recall from \S\ref{sec:tbdd} that we compactify in two steps: first we
make a space totally bounded, then we \emph{complete} it.  We've
already defined total boundedness; let's look now at completeness, or
``filling in the holes''.  What does it mean for a space to ``have a
hole''?  That is, looking at a gauge space $X$, how can we
characterize a ``point which should exist in $X$, but doesn't''?  One
obvious way is to give its distances from all the other points of $X$.

\begin{defn}
  A \textbf{Cauchy point} of a gauge space $X$ consists of, for each
  gauge metric $d$, a function $\xi_d\maps X\to [0,\infty)$ such that
  \begin{enumerate}
  \item $d(x,y)+\xi_d(y) \ge \xi_d(x)$ for each $d$ and each $x,y\in X$,
  \item $\xi_d(x)+\xi_d(y) \ge d(x,y)$ for each $d$ and each $x,y\in X$, 
  \item $\inf_{x\in X} \xi_d(x) = 0$ for each $d$ (\emph{locatedness}), and
  \item For any two gauge metrics $d_1$ and $d_2$, there exists a
    gauge metric $d_3$ such that $d_3\ge \max(d_1,d_2)$ and $\xi_{d_3}
    \ge \max(\xi_{d_1},\xi_{d_2})$.
  \end{enumerate}
\end{defn}

We think of $\xi_d(x)$ as ``$d(\xi,x)$''.  The first two conditions are
then just two versions of the triangle inequality, and can be jointly
rephrased as the `reverse triangle inequality'
\[|\xi_d(x) - d(x,y)| \le \xi_d(y).
\]
The locatedness condition can be phrased as ``every ball $B_d(\xi,\ep)$
contains a point of $X$''.  This ensures that the `new point' $\xi$ is
`right next to $X$'; we don't want to think about potential points
that are sitting far away from $X$.  And the fourth condition ensures
that the filteredness condition on gauge metrics carries over to the
functions $\xi_d$, since they are supposed to be like values of the
metrics.

We note that in conjunction with the first three conditions, the
fourth implies the following stronger version of itself.

\begin{lem}\label{thm:fourth}
  If $\xi$ is a Cauchy point and $d_1$ and $d_2$ are gauge metrics with
  $d_1\le d_2$, then $\xi_{d_1}\le \xi_{d_2}$.
\end{lem}
\begin{proof}
  Choose $d_3$ as assumed in the fourth condition.  Now for any
  $\ep>0$, use locatedness to choose an $x\in X$ with
  $\xi_{d_3}(x)<\ep$, hence also $\xi_{d_1}(x)<\ep$ and $\xi_{d_2}<\ep$.
  Thus for any $y\in X$,
  \begin{align}
    \xi_{d_1}(y)
    &\le d_1(x,y) + \xi_{d_1}(x)\\
    &\le d_2(x,y) + \ep\\
    &\le \xi_{d_2}(y) + \xi_{d_2}(x) + \ep\\
    &\le \xi_{d_2}(y) + 2\ep.
  \end{align}
  Since this is true for all $\ep>0$, we have $\xi_{d_1} \le \xi_{d_2}$.
\end{proof}


If $z\in X$, then we have a Cauchy point called \zhat\ defined by
$\zhat_d(x) = d(z,x)$.  We say a Cauchy point is \textbf{represented
  by $z$} if it is equal to $\zhat$.  We say that a gauge space is
\textbf{complete} if every Cauchy point is represented (by some
point).

\begin{lem}
  A Cauchy point $\xi$ is represented by $z$ iff $\xi_d(z)=0$
  for all $d$.
\end{lem}
\begin{proof}
  Clearly if $\xi$ is represented by $z$, then $\xi_d(z) = d(z,z)=0$.
  Conversely, if $\xi_d(z)=0$, then for any $x$ we have
  \[\xi_d(x) \le d(x,z) + \xi_d(z) = d(x,z)\]
  and
  \[\xi_d(x) = \xi_d(x) + \xi_d(z) \ge d(x,z)\]
  and thus $\xi_d(x) = d(x,z)$, so $\xi = \zhat$.
\end{proof}

It's easy to come up with silly examples of non-represented Cauchy
points.
\begin{itemize}
\item If $X=\lR^2 \setminus \{0\}$, then $\xi_d(x) = d(0,x)$ is a
  Cauchy point which is not represented (it should be represented by
  0).  Thus $\lR^2\setminus \{0\}$ is not complete.
\item More generally, if $X$ is any gauge space and $y\in X$, then on
  $Y=X\setminus \{y\}$ the collection of functions $\xi_d(z) = d(y,z)$
  (the distance taken in $X$) defines a non-representable Cauchy point
  of $Y$, so $Y$ is not complete.
\end{itemize}

In these cases, it is trivial to see how to `complete' the incomplete
spaces; add back in the missing point!  In general, what we can do is
use the Cauchy points \emph{themselves} to stand for the missing
points.

Let \Xhat\ be the set of Cauchy points in $X$.  Each point of $X$
gives a Cauchy point \xhat, so we have an inclusion $X\to \Xhat$.  We
want to make $\Xhat$ a gauge space; that is, we need to somehow extend
the metrics on $X$ to $\Xhat$.

By definition, a Cauchy point knows what its distances are to any
point of $X$, but not its distances to \emph{another} Cauchy point.
However, each Cauchy point is approximated arbitrarily closely by
points of $X$, so we should be able to get pretty close to the
distance between two Cauchy points $\xi$ and $\ze$ by going first from
$\xi$ to some point of $X$, then to $\ze$.

This motivates the following definition.  For each metric $d$ on $X$,
define a metric \dhat\ on \Xhat\ by setting
\[\dhat(\xi,\ze) = \inf_{x\in X} \big(\xi_d(x) + \ze_d(x)\big).\]
The locatedness of Cauchy points implies $\dhat(\xi,\xi)=0$.  Symmetry is
clear.  For transitivity, we have
\begin{align*}
  \dhat(\xi,\ze) + \dhat(\ze,\chi)
  &= \inf_x \big( \xi_d(x) + \ze_d(x) \big) + \inf_y \big(\ze_d(y) + \chi_d(y)\big)\\
  &= \inf_{x,y} \big( \xi_d(x) + \ze_d(x) + \ze_d(y) + \chi_d(y)\big)\\
  &\ge \inf_{x,y} \big( \xi_d(x) + d(x,y) + \chi_d(y)\big)\\
  &\ge \inf_{x} \big( \xi_d(x) + \chi_d(x) \big)\\
  &= \dhat(\xi,\chi).
\end{align*}
Finally, \autoref{thm:fourth} implies that if $d\ge d'$, then $\dhat
\ge \widehat{d'}$.  Therefore, the collection of metrics $\{\dhat\}$
forms a gauge on \Xhat, since given $\widehat{d_1}$ and
$\widehat{d_2}$ we can find $d_3$ with $d_3\ge d_1$ and $d_3\ge d_1$,
hence $\widehat{d_3}\ge\widehat{d_1}$ and $\widehat{d_3}\ge
\widehat{d_1}$.

\begin{lem}[Yoneda]\label{thm:yoneda}
  For any $z\in X$ and $\xi\in\Xhat$, and any $X$-metric $d$, we have
  $\dhat(\zhat,\xi) = \xi_d(z)$.  In particular, for $z,w\in X$ we
  have $\dhat(\zhat,\what) = d(z,w)$.
\end{lem}
\begin{proof}
  By the triangle inequality for \xi, for any $x$ we have
  \[ \zhat_d(x) + \xi_d(x) = d(z,x) + \xi_d(x) \ge \xi_d(z) \]
  so that $\dhat(\zhat,\xi) \ge \xi_d(z)$.  But on the other hand, we
  also have
  \[ \dhat(\zhat,\xi) \le \zhat_d(z) + \xi_d(z) = \xi_d(z). \qedhere \]
\end{proof}

Thus, the function $\widehat{(-)}\maps X\to \Xhat$ sending $z$ to
$\zhat$ is an ``embedding'' --- though it may not be injective.  In
fact, we have:

\begin{lem}\label{ex:cplt-sep}
  $\widehat{(-)}\maps X\to \Xhat$ is injective if and only if $X$ is
  separated.  Moreover, \Xhat\ is always separated.
\end{lem}
\begin{proof}
  It's easy to see that $\xhat=\yhat$ if and only if $x\approx y$,
  which proves the first statement.
  The second is an exercise.
\end{proof}

The locatedness of Cauchy points also implies that $X$ is
\textbf{dense} in \Xhat, i.e.\ that $\xi\approx X$ for any
$\xi\in\Xhat$.

Finally, we show \Xhat\ is complete.  Suppose $\Xi$ is a Cauchy point of
\Xhat, and define a Cauchy point $\xi$ of $X$ by
\[\xi_d(x) = \Xi_{\dhat}(\xhat).
\]
Now we check the axioms for $\xi$ to be a Cauchy point of $X$.
\begin{enumerate}
\item $d(x,y) + \xi_d(y) = \dhat(\xhat,\yhat) + \Xi_{\dhat}(\yhat) \ge
  \Xi_{\dhat}(\xhat) = \xi_d(x)$.
\item $\xi_d(x) + \xi_d(y) = \Xi_{\dhat}(\xhat) + \Xi_{\dhat}(\yhat) \ge
  \dhat(\xhat,\yhat) = d(x,y)$
\item Since $\inf_{\ze\in\Xhat} \Xi_{\dhat}(\ze) = 0$, for any $\ep>0$ we
  have an $\ze\in\Xhat$ with $\Xi_{\dhat}(\ze) <\frac\ep2$.  And since $\ze$
  is a Cauchy point of $X$, we have an $x\in X$ with $\ze_d(x)
  <\frac\ep2$.  Thus,
  \[ \textstyle \xi_d(x) = \Xi_{\dhat}(\xhat) \le \dhat(\xhat,\ze) + \Xi_{\dhat}(\ze) =
  \ze_d(x) + \Xi_{\dhat}(\ze) \le \frac\ep2 + \frac\ep2 = \ep \]
  Since this is true for all \ep, we have $\inf_{x\in X} \xi_d(x) = 0$.
\item If $d \ge d'$, then $\dhat \ge \dhat'$, so
  \[ \xi_d(x)  = \Xi_{\dhat}(\xhat) \ge \Xi_{\dhat'}(\xhat) = \xi_{d'}(x) \]
\end{enumerate}

Now I claim that $\Xi=\xihat$.  We know it's enough to show $\Xi_{\dhat}(\xi) = 0$ for
all $d$.  To show this, fix an $\ep>0$ and find an $x$ with
$\xi_d(x)<\ep$.  Then we have
\begin{equation}
  \Xi_{\dhat}(\xi) \;\le\; \Xi_{\dhat}(\xhat) + \dhat(\xhat,\xi)
  \;=\; \xi_d(x) + \xi_d(x)
  \;<\; 2\ep.
\end{equation}
Since this is true for all $\ep>0$, we must have $\Xi_d(\xi) = 0$, hence
$\Xi=\xihat$.

We can now give a definition of compactness.

\begin{defn}
  A gauge space is \textbf{compact} if it is totally bounded and
  complete.
\end{defn}

If you've seen compactness before, you may have seen a different
definition.  Ours is a bit of a cheat, but it is correct.

Finally, remember that we want to construct compactifications by
finding an equivalent totally bounded gauge, then completing it.  For
this to work, we need the following.

\begin{prop}
  The completion of a totally bounded space is totally bounded.
\end{prop}
\begin{proof}
  Exercise.
\end{proof}

\section{Exercises on completeness}
\label{sec:exerc-compl}

\begin{ex}
  Prove that the completion of a totally bounded gauge space is
  totally bounded (hence compact).
\end{ex}

\begin{ex}
  Prove that \Xhat is always separated.
\end{ex}

\begin{ex}
  Prove that a Cauchy point $\xi$ is represented by $z$ if and only if
  for every gauge metric $d$ and every $\ep>0$, there is a point $x$
  with $\xi_d(x)\le \ep$ and $d(x,z)\le\ep$.
\end{ex}


\begin{ex}
  Prove that \lR\ is complete with the usual metric $d(x,y)=|x-y|$,
  and that it is uniformly equivalent to the completion of its
  subspace \lQ.  \emph{(Hint: you'll need Dedekind-completeness or
    Cauchy-completeness of \lR.)}
\end{ex}

\begin{ex}
  Prove that \lR\ is \emph{not} complete with the metric \[d^*(x,y) =
  \big|\arctan x-\arctan y\,\big|\] (which is topologically equivalent to the
  usual metric).
\end{ex}


\begin{ex}
  Prove that if $X$ is complete, then every Cauchy sequence converges
  to some point.
\end{ex}

\begin{ex}
  Suppose $X$ is a metric space, i.e.\ a gauge space with only one
  gauge metric.  Prove that if every Cauchy sequence in $X$ converges
  to some point, then $X$ is complete.
\end{ex}

\begin{exstar}
  Find an example of a non-complete gauge space in which every Cauchy
  sequence converges.
\end{exstar}

\begin{exstar}\label{ex:tbdd-cauchy}
  Suppose $X$ is a totally bounded gauge space and $\xi,\ze$ are two
  Cauchy points of $X$.  Prove that if $\xi\approx A \Leftrightarrow
  \ze \approx A$ for all $A\subseteq X$, then $\xi=\ze$.  (Of course,
  $\xi\approx A$ means that $\inf_{a\in A} \xi_d(a) = 0$ for all
  $X$-metrics $d$.)
\end{exstar}

\begin{exstar}\label{ex:cplt-xtn}
  Prove that if $f:X\to Y$ is uniformly continuous, then it extends
  uniquely to a uniformly continuous function $\fhat:\Xhat\to\Yhat$.
\end{exstar}

\begin{exstar}\label{ex:cplt-unif}
  Prove that completeness is a uniform property: if $X$ and $Y$ are
  uniformly isomorphic and $X$ is complete, then so is $Y$.
\end{exstar}

\section{Exercises on compactness}
\label{sec:ex-cpt}

\begin{ex}\label{thm:metric-seqcpt}
  Prove that a metric space is compact if and only if every sequence
  has a convergent subsequence.
\end{ex}

\begin{ex}\label{ex:inf-cube}
  Let $X$ be $[0,1]^\om$; that is, the set of infinite sequences of
  real numbers in $[0,1]$.  Equip it with the gauge defined by the
  metrics $d_N(x,y) = \max_{1\le i\le N} |x_i-y_i|$ for
  $N=0,1,2,\dots$ (this is the restriction of the gauge on $\lR^\om$
  considered in class).  Prove that $X$ is compact.
\end{ex}

\begin{ex}\label{ex:cpt-union}
  Let $K,L$ be compact subsets of $X$ (with respect to their induced
  gauges).  Prove that $K\cup L$ is compact.  Conclude by induction
  that the union of any finite number of compact subsets of $X$ is
  compact.
\end{ex}

\begin{ex}
  Prove that a gauge space $X$ is compact if and only if its separated
  quotient from \autoref{ex:sep-quot} is compact.
\end{ex}



\begin{exstar}[Tychonoff's Theorem]
  Let $\{X_\al\}_{\al\in A}$ be a family of gauge spaces and $X =
  \Pi_{\al\in A} X_\al$ their cartesian product.  By definition, that
  means the elements of $X$ are families $(x_\al)_{\al\in A}$ where
  $x_\al\in X_\al$ for all \al.  Each metric $d$ on some $X_\al$
  defines a metric on $X$ via $d_\al(x,y) = d(x_\al,y_\al)$, and the
  collection of all these metrics generates a gauge on $X$ (by taking
  finite maxes).  For each of completeness, total boundedness, and
  compactness, prove that if each $X_\al$ has the property in
  question, so does $X$.
\end{exstar}

\section{The one-point compactification}
\label{sec:one-point}

Recall our procedure for constructing compactifications: first find a
topologically equivalent, totally bounded gauge, then complete it.  In
the next section we'll construct the Stone-\v{C}ech compactification.
But first, as a warm-up, let's try adding exactly one point at
infinity, as in stereographic projection.

It turns out that we need some condition on $X$ for this to work.
Define a gauge space to be \textbf{locally compact} if for every point
$x$ and open ball $B$ containing $x$, there exists a compact closed ball $\Bbar$
with $x\in\Bbar\subseteq B$.  Here a closed ball is a set of the form
\[\Bbar = \overline{B}_d(x,\ep) = \{ y \mid d(x,y)\le \ep\}.
\]

Let $X$ be a locally compact, but noncompact, gauge space.  For every
compact $K\subseteq X$ and gauge metric $d$, denote by $\neg K$ the
complement of $K$ in $X$, and define $d_K$ by
\begin{equation*}
  d_K(x,y) =
  \min\Big(d(x,y),\; d(x,\neg K)+d(\neg K,y)\Big).
\end{equation*}
By \autoref{ex:collapsing-metric}, $d_K$ is a metric.  Clearly if
$d\ge d'$ and $K\supset K'$, then $d_K\ge d'_{K'}$.  Thus, using
\autoref{ex:cpt-union}, we conclude that the collection of all the
$d_K$, as $d$ and $K$ vary, is a new gauge on $X$.

\begin{thm}\label{thm:1pt-tb-eqv}
  When $X$ is locally compact, this new gauge is totally bounded and
  topologically equivalent to the original gauge.
\end{thm}
\begin{proof}
  Given any $d_K$ and $\ep>0$, choose a finite $(d,\ep)$-sized cover
  of $K$ (which exists since $K$ is totally bounded).  Then together
  with $\neg K$, which is certainly $(d_K,\ep)$-sized (all points in
  it are $d_K$-distance 0 from each other), we have a
  $(d_K,\ep)$-sized cover of $X$ in the new gauge.  Thus, the new
  gauge is totally bounded.

  For topological equivalence, since $d_K(x,y)\le d(x,y)$, every ball
  in the new gauge trivially contains a ball in the old gauge.
  Conversely, suppose given $x$ and an open ball $x\in B$ in the old
  gauge.  Choose $x\in \overline{B}_d(x,\ep)\subseteq B$ as in the
  definition of local compactness, and consider the ball
  $B_{d_K}(x,\ep)$ in the new gauge, where $K=\overline{B}_d(x,\ep)$.
  Since $d(x,\neg K)\ge \ep$, we have
  \[B_{d_K}(x,\ep) = B_{d}(x,\ep) \subseteq B.\] This proves topological
  equivalence.
\end{proof}

\begin{thm}
  There is exactly one non-represented Cauchy point in this new gauge
  space.
\end{thm}
\begin{proof}
  For existence, set $\xi_{d_K}(x) = d(x,\neg K)$.  It is easy to
  check that this defines a non-represented Cauchy point.  For
  uniqueness, suppose that $\xi$ is a Cauchy point.  Pick $K\subseteq
  X$ compact.  If $\inf_{k\in K} \xi_{d_L}(k) = 0$ for all $d_L$, then
  $\xi$ would induce a Cauchy point of $K$, and would thus be
  represented since $K$ is compact.  Therefore, there must be some
  $X$-metric $d$, compact $L\subseteq X$, and $\de>0$ such that
  $\xi_{d_L}(k)>\de$ for all $k\in K$.

  Now since $\xi$ is a Cauchy point, for any $\ep>0$, there is some
  $z\in X$ such that both $\xi_{d_L}(z)<\ep$ and $\xi_{d_K}(z)<\ep$.
  But by the above, as long as $\ep<\de$, this $z$ cannot be in $K$,
  so that $d(z,\neg K) = 0$.  Thus, for any $x\in X$ we have
  \[ d_K(x,z) \le d(x,\neg K).
  \]
  We therefore have
  \begin{equation}
    \xi_{d_K}(x) \le d_K(x,z) + \xi_{d_K}(z) \le d(x,\neg K) + \ep
  \end{equation}
  As this is true for all $\ep>0$, we have $\xi_{d_K} \le d(-,\neg
  K)$.  In particular, we have $\xi_{d_K}(y) \le d(y,\neg K) = 0$ for
  any $y\notin K$.  Thus, fixing such a $y$, for any $x\in X$ we have
  \begin{equation}
    d(x,\neg K) = d_K(x,y) \le \xi_{d_K}(x) + \xi_{d_K}(y) = \xi_{d_K}(x).
  \end{equation}
  and thus $\xi_{d_K} = d(-,\neg K)$.
\end{proof}


We call the completion of this new gauge the \textbf{one-point
  compactification} of $X$, since it contains only one point in
addition to $X$.  We write it as $X_\infty$ or $\al X$.
For example:
\begin{itemize}
\item $S^n$ is topologically isomorphic to the one-point
  compactification of $\lR^n$.
\item The one-point compactification of the infinite cylinder is the
  pinched torus.
\end{itemize}

\section{The Stone-\v Cech Compactification}
\label{sec:stone-cech}

Finally, we're ready to consider the question: what is the \emph{best}
compactification?  Since we compactify by finding a totally bounded
gauge and then completing, this question essentially boils down to
finding the best totally bounded gauge which is topologically
equivalent to the one we started with.  We proceed as follows.

Let $n\in\lN$, let $\ph\maps X\to [0,1]^n$ be any continuous function, and define a
metric on $X$ by
\[d_\ph(x,y) = d_n(\ph(x),\ph(y)),
\]
where $d_n(a,b) = \max_{1\le i\le n} |a_i-b_i|$ is one of the
canonical metrics on $[0,1]^n$.  The function $d_\ph$ is clearly a
metric on $X$, and the collection of all such metrics defines a gauge
on $X$.  (For filteredness, given $\ph_1\maps X\to [0,1]^n$ and
$\ph_2\maps X\to [0,1]^m$ let $\ph_3=(\ph_1,\ph_2)\maps X\to
[0,1]^{n+m}$.)

\begin{thm}\label{thm:stonecechgauge}
  This gauge is totally bounded and topologically equivalent to the
  original one.
\end{thm}
\begin{proof}
  Suppose first we are given $x\in X$, $\ph\maps X\to [0,1]^n$ and
  $\ep>0$; we want to show that $B_{d_\ph}(x,\ep)$ contains a ball in
  the original gauge.  But
  \[B_{d_\ph}(x,\ep)
  = \{y \mid d_n(\ph(x),\ph(y))<\ep\}
  = \ph\inv(B(\ph(x),\ep))
  \]
  so this follows since \ph\ is continuous.

  Now suppose we are given $x\in X$, $d$, and $\ep>0$, and we want to
  show that $B_d(x,\ep)$ contains a ball in the new gauge.  But by
  \autoref{ex:dist-cts}, the function $\psi(y) = \max(d(x,y),1)$ is a
  continuous map $X\to [0,1]$, and it is easy to check that if
  $d_\psi(x,y)<\ep$, then $d(x,y)<\ep$.

  Finally, to show the new gauge is totally bounded, let $\ep>0$ and
  $\ph\maps X\to [0,1]^n$ be given.  Cover $[0,1]^n$ by finitely many
  \ep-sized sets $A_j$; then the sets $\ph\inv(A_j)$ are a finite
  $(d_\ph,\ep)$-sized cover of $X$.
\end{proof}

\begin{cor}\label{cor:eqv-tbdd}
  Every gauge space is topologically equivalent to a totally bounded
  one.
\end{cor}

Call this new gauge $\cG_\be$.  The completion of $(X,\cG_\be)$ is the
\textbf{Stone-\v Cech compactification} of $X$; we write it as $\be
X$.  Since $\cG_\be$ is totally bounded, $\be X$ is compact.  Note
that the inclusion $X\to \be X$ is uniformly continuous when $X$ has
the new gauge $\cG_\be$, but only continuous when $X$ has the original
gauge.

The points of $\be X$ are, of course, the Cauchy points of $(X,\cG_\be)$.
Such a Cauchy point consists of, for each continuous $\ph\maps X\to
[0,1]^n$, a function $\xi_\ph\maps X\to [0,\infty)$, such that
\begin{itemize}
\item $\xi_\ph(x) + d_n(\ph(x),\ph(y)) \ge \xi_\ph(y)$,
\item $\xi_\ph(x) + \xi_\ph(y) \ge d_n(\ph(x),\ph(y))$,
\item $\inf_x \xi_\ph(x) = 0$, and
\item If $d_\ph \le d_\psi$, then $\xi_\ph \le \xi_\psi$.
\end{itemize}
We can reformulate this in a convenient way as follows.  For each
continuous $\ph\maps X\to [0,1]$ and $\ep>0$, define
\begin{align*}
  A_{\ph,\ep} &= \inf_{\xi_\ph(x)\le \ep} \ph(x)\\
  B_{\ph,\ep} &= \sup_{\xi_\ph(x)\le \ep} \ph(x).
\end{align*}
Clearly $A_{\ph,\ep}\le B_{\ph,\ep}$, and moreover by the triangle
inequality, $B_{\ph,\ep} - A_{\ph,\ep} \le 2\ep$.  Therefore, by a
standard property of real numbers, the intersection $\bigcap_\ep
[A_{\ph,\ep},B_{\ph,\ep}]$ consists of a single point.  Call that
point $\ph(\xi)$.  More generally, if $\ph=(\ph_1,\dots,\ph_n)\maps
X\to [0,1]^n$, then we write $\ph(\xi) =
(\ph_1(\xi),\dots,\ph_n(\xi))$.

Of course, a point $x\in X$ induces a Cauchy point of $(X,\cG_\be)$,
which we write as $\xchk$ to avoid confusion with $\xhat$, the induced
Cauchy point of $X$ in its original gauge.  It is easy to see that we
have $\ph(\xchk) = \ph(x)$ for any such $x$.

We have shown that any Cauchy point of $(X,\cG_\be)$ gives a way to
`evaluate' functions \ph.  We now want to show that the Cauchy point
$\xi$ is determined by the operation $\ph \mapsto \ph(\xi)$.
Specifically, we want to show that $\xi_\ph(y) = d_n(\ph(y),\ph(\xi))$ for
any $y$.  To show this, let $\ep>0$ and choose some $x$ with
$\xi_\ph(x)\le \ep$ (possible by locatedness of $\xi$).  Then the
reverse triangle inequality gives
\[\Big|\xi_\ph(y) - d_n(\ph(x),\ph(y))\Big| \le \ep.\]
Now, by definition of $x$ and the filteredness of a Cauchy point,
we have $\xi_{\ph_i}(x)\le \ep$ for all coordinates $1\le i\le n$, and
therefore $\ph_i(x) \in [A_{\ph_i,\ep}, B_{\ph_i,\ep}]$.  Hence, by
definition of $\ph_i(\xi)$, we have $|\ph_i(x)-\ph_i(\xi)|\le 2\ep$, and
thus $d_n(\ph(x),\ph(\xi)) \le 2\ep$.  It follows that
\[\Big|\xi_\ph(y) - d_n(\ph(\xi),\ph(y))\Big| \le 3\ep.\]
But since this holds for all $\ep>0$, and the left-hand side is
independent of \ep, it must be zero; thus $\xi_\ph(y) =
d_n(\ph(\xi),\ph(y))$ as desired.

Now suppose we have an ``evaluation'' operation taking any continuous
function $\ph:X\to [0,1]$ to a number $\ph(\xi)\in[0,1]$.  As before,
if $\ph=(\ph_1,\dots,\ph_n)\maps X\to [0,1]^n$, then we write
$\ph(\xi) = (\ph_1(\xi),\dots,\ph_n(\xi))$.  Then if we define
$\xi_\ph(y) = d_n(\ph(y),\ph(\xi))$, we get a collection of functions
satisfying the first two conditions to be a Cauchy point, and the
definition of $\ph(\xi)$ when $n>1$ immediately implies the fourth
condition.  Thus, to obtain a Cauchy point in this way we only need to
ensure locatedness.  We summarize this as follows:

\begin{lem}\label{thm:cauchy-beta}
  To give a Cauchy point of $(X,\cG_\be)$ is equivalent to giving, for
  every continuous $\ph:X\to [0,1]$, a number $\ph(\xi)\in [0,1]$, such
  that
  \begin{quote}
    For any finite family $(\ph_i)_{1\le i\le n}$ and
    any $\ep>0$, there exists a point $x\in X$ such that $|\ph_i(x) -
    \ph_i(\xi)|<\ep$ for all $i$.
  \end{quote}
\end{lem}

In other words, the points of $\be X$ are the operations $\ph\mapsto
\ph(\xi)$ which are arbitrarily closely approximable by (evaluation at)
points of $X$.  This gives another way of viewing them as `virtual
points' of the space $X$.

We note in passing that \autoref{thm:cauchy-beta} can be reformulated
in the following concise way.  Let $C(X,[0,1])$ denote the set of
continuous functions $X\to [0,1]$, and consider the space
$[0,1]^{C(X,[0,1])}$ with the pointwise gauge as in
\autoref{thm:pointwise}.  There is a continuous function $e:X\to
[0,1]^{C(X,[0,1])}$ defined by $e(x)_\ph = \ph(x)$, and the points of
$\be X$ can be identified with the points $\xi\in [0,1]^{C(X,[0,1])}$ such
that $e(X) \approx \xi$.

The following lemma will also be useful.

\begin{lem}\label{thm:becauchy-le}
  Let $\xi$ be a Cauchy point of $(X,\cG_\be)$.
  \begin{enumerate}
  \item If $\ph,\psi\maps X\to [0,1]$ are continuous with $\ph \le
    \psi$, then $\ph(\xi)\le\psi(\xi)$.\label{item:becle1}
  \item If $\psi\maps [0,1]^n \to [0,1]$ and $\ph\maps X\to [0,1]^n$
    are continuous, then we have $\psi(\ph(\xi)) = (\psi\ph)(\xi)$.
  \end{enumerate}
\end{lem}
\begin{proof}
  For the first statement, observe that for any $\ep>0$ there is an
  $x$ with $\xi_{\ph,\psi}(x)<\ep$, and hence both $\xi_\ph(x)<\ep$ and
  $\xi_\psi(x)<\ep$.  Therefore $\ph(x)\in [A_{\ph,\ep},B_{\ph,\ep}]$
  and $\psi(x)\in [A_{\psi,\ep},B_{\psi,\ep}]$.  Since
  $\ph(x)\le\psi(x)$, we have $A_{\ph,\ep}\le B_{\psi,\ep}$, from
  which $\ph(\xi)\le \psi(\xi)$ follows (taking $\ep\to 0$).

  Now consider the second statement.  Fix an $\ep>0$ until further
  notice.  By definition of $(\psi\ph)(\xi)$, there is a $\de>0$ such
  that if $\xi_{\psi\ph}(x)<\de$, then $|\psi\ph(x)-\psi\ph(\xi)|<\ep$.
  And since $\psi$ is continuous, there is a $\de'>0$ such that if
  $d_n(\ph (x), \ph(\xi))<\de'$, then $|\psi(\ph(x)) -
  \psi(\ph(\xi))|<\ep$.  But by definition of $\ph(\xi)$, there is then a
  $\de''$ such that if $\xi_\ph(x)<\de''$, then $d_n(\ph (x),
  \ph(\xi))<\de'$, and hence also $|\psi(\ph(x)) - \psi(\ph(\xi))|<\ep$.

  Now, since $\xi$ is a Cauchy point, there is an $x$ such that $\xi_{\ph,
    \psi\ph}(x)< \min(\de,\de'')$.  It follows that we have both
  $|\psi\ph(x)-\psi\ph(\xi)|<\ep$ and $|\psi(\ph(x)) -
  \psi(\ph(\xi))|<\ep$, and hence, by the triangle inequality,
  $|(\psi\ph)(\xi) - \psi(\ph(\xi))|<2\ep$.  Since this is true for all
  $\ep>0$, we must have $(\psi\ph)(\xi) = \psi(\ph(\xi))$, as desired.
\end{proof}

Finally, the following theorem is one way to express the idea that the
Stone-\v Cech compactification the \emph{best} compactification.

\begin{thm}[Extension Theorem]\label{thm:extension}
  Let $g\maps X\to Y$ be any continuous map, where $Y$ is compact.
  Then there exists a continuous map $\gbar\maps \be X\to Y$ such that
  for all $x\in X$, we have $\gbar(\xchk) \approx g(x)$.  Moreover, if
  $h:\be X \to Y$ is any other continuous map with this property, then
  for all $\xi\in \be X$ we have $\gbar(\xi)\approx h(\xi)$.
\end{thm}

Let us consider how to prove this theorem.  Given $\xi\in\be X$,
represented as above by the operation $\ph\mapsto \ph(\xi)$, we would
like to define a Cauchy point $\gbar(\xi)$ of $Y$ by
\begin{equation}
  (\gbar(\xi))_d(y) = \psi_{d,g,y}(\xi)\label{eq:gbardef}
  \quad\text{where}\quad
  \psi_{d,g,y}(x) = d(g(x),y).
\end{equation}
Note that $\psi_{d,g,y}$ is continuous since $g$ and $d$ are.
Unfortunately, it will not generally take values in $[0,1]$.
We can remedy this by using instead
\begin{equation}
  (\gbar(\xi)')_d(y) = \psi'_{d,g,y}(\xi)
  \quad\text{where}\quad
  \psi'_{d,g,y}(x) = \min(d(g(x),y),1).
\end{equation}
but now we cannot expect $\gbar(\xi)'$ to be a Cauchy point; at large
distances it will fail the triangle inequality.  The solution to this
conundrum is the following lemma.

\begin{lem}\label{thm:cauchy-smalldet}
  Suppose $X$ is a gauge space and we have a family of subsets $A_d
  \subseteq X$, one for each $X$-metric $d$, and also functions
  $\xi_d:A_d\to [0,\infty)$ such that
  \begin{enumerate}
  \item $d(x,y)+\xi_d(y) \ge \xi_d(x)$ for each $d$ and each $x,y\in A_d$.
  \item $\xi_d(x)+\xi_d(y) \ge d(x,y)$ for each $d$ and each $x,y\in A_d$.
  \item $\inf_{x\in A_d} \xi_d(x) = 0$ for each $d$.
  \item For any two gauge metrics $d_1$ and $d_2$, there exists a
    gauge metric $d_3$ such that $A_{d_3}\subseteq A_{d_1} \cap
    A_{d_2}$ and $d_3\ge \max(d_1,d_2)$ and $\xi_{d_3} \ge
    \max(\xi_{d_1},\xi_{d_2})$.
  \end{enumerate}
  Then there exists a unique Cauchy point \ze of $X$ such that
  $\ze_d(x) = \xi_d(x)$ for all gauge metrics $d$ and all $x\in A_d$.
\end{lem}
\begin{proof}
  I claim that if \ze is as asserted, then we must have
  \begin{equation}
    \ze_d(x) = \inf_{a\in A_d} (d(x,a) + \xi_d(a))\label{eq:cauchy-smalldet}
  \end{equation}
  for all $d$.  It is easy to check that this definition gives a
  Cauchy point, so it suffices to show that any such \ze must be equal
  to this one.  The triangle inequality immediately gives
  \begin{equation}
    \ze_d(x) \le d(x,a) + \ze_d(a) = d(x,a) + \xi_d(a)
  \end{equation}
  so it remains to prove the direction $\ge$
  of~\eqref{eq:cauchy-smalldet}.  For this, let $\ep>0$ and let $a_0\in
  A_d$ be such that $\xi_d(a_0) < \ep$.  Then
  \begin{align}
    \inf_{a\in A_d} (d(x,a) + \xi_d(a))
    &\le d(x,a_0) + \xi_d(a_0)\\
    &\le \ze_d(x) + \ze_d(a_0) + \xi_d(a_0)\\
    &= \ze_d(x) + \xi_d(a_0) + \xi_d(a_0)\\
    &\le \ze_d(x) + 2\ep.
  \end{align}
  Taking $\ep\to 0$ we obtain the desired inequality.
\end{proof}

\begin{proof}[Proof of \autoref{thm:extension}]
  Given $\xi\in\be X$, for any $Y$-metric $d$ define
  \[ A_d = \setof{ y\in Y | \psi_{d,g,y}(\xi) < 1 }. \]
  where $\psi_{d,g,y}(x) = d(g(x),y)$.
  Now as suggested above, for $y\in A_d$ define
  \begin{equation}
    (\gbar(\xi)')_d(y) = \psi'_{d,g,y}(\xi)\label{eq:gbardefprime}
    \quad\text{where}\quad
    \psi'_{d,g,y}(x) = \min(d(g(x),y),1).
  \end{equation}
  We will show that $\gbar(\xi)'$ with these sets satisfies the hypotheses of
  \autoref{thm:cauchy-smalldet}.

  Let $y_1,y_2\in A_d$, and let $\ep>0$ be smaller than $1 -
  \psi_{d,g,y_i}(\xi)$ for $i=1,2$.  Now choose an $x$ with
  $|\psi_{d,g,y_i}(x)-\psi_{d,g,y_i}(\xi)|<\frac\ep2$ for $i=1,2$;
  then we also have $\psi_{d,g,y_i}(x) < 1$ and hence
  $\psi'_{d,g,y_i}(x) = \psi_{d,g,y_i}(x) = d(g(x),y_i)$.  Thus we can compute
  \begin{multline*}
    (\gbar(\xi)')_d(y_1)  + d(y_1,y_2)
    = \psi'_{d,g,y_1}(\xi) + d(y_1,y_2)\\
    = \psi_{d,g,y_1}(\xi) + d(y_1,y_2)
    \ge \psi_{d,g,y_1}(x) - \frc\ep2 + d(y_1,y_2)
    = d(g(x),y_1) + d(y_1,y_2) - \frc\ep2\\
    \ge d(g(x),y_2) - \frc\ep2
    = \psi_{d,g,y_2}(x) - \frc\ep2
    \ge \psi_{d,g,y_2}(\xi) - \ep\\
    = \psi'_{d,g,y_2}(\xi) - \ep
    = (\gbar(\xi))_d(y_2) - \ep.
  \end{multline*}
  Since this is true for all $\ep>0$, we obtain
  \[(\gbar(\xi)')_d(y_1) + d(y_1,y_2) \ge (\gbar(\xi)')_d(y_2).\]
  Similarly, with the same \ep\ and $x$, we have
  \begin{multline}
    (\gbar(\xi)')_d(y_1) + (\gbar(\xi)')_d(y_2)
    = \psi'_{d,g,y_1}(\xi) + \psi'_{d,g,y_2}(\xi)
    = \psi_{d,g,y_1}(\xi) + \psi_{d,g,y_2}(\xi)\\
    \ge \psi_{d,g,y_1}(x) + \psi_{d,g,y_2}(x) - \ep
    = d(g(x),y_1) + d(g(x),y_2) - \ep
    \ge d(y_1,y_2) - \ep.
  \end{multline}
  Again, taking $\ep\to 0$ we obtain
  \[(\gbar(\xi)')_d(y_1) + (\gbar(\xi)')_d(y_2) \ge d(y_1,y_2).
  \]
  Thus $\gbar(\xi)$ satisfies the first two conditions of
  \autoref{thm:cauchy-smalldet}.  For the third, we want to show that
  given any $Y$-metric $d$ and any $\ep>0$, there exists a $y\in A_d$
  with $(\gbar(\xi)')_{d}(y)<\ep$.  We may assume $\ep<1$.  Now since
  $Y$ is totally bounded, we can cover it by finitely many balls
  $B_d(z_j,\frc\ep2)$.  Since $\xi$ is a Cauchy point of
  $(X,\cG_\be)$, we can choose an $x\in X$ such that for all $j$ we
  have $|\psi'_{d,g,z_j}(x) - \psi'_{d,g,z_j}(\xi)|<\frc\ep2$.
  Now since our balls covered $Y$, we must have $g(x) \in
  B_d(z_j,\frc\ep2)$ for some $j$; let $y$ be that $z_j$.  Then since
  $\psi_{d,g,y}(x) = d(g(x),y) < \frc\ep2$, we have $\psi'_{d,g,y}(x)
  = \psi_{d,g,y}(x) = d(g(x),y)$.  Thus:
  \begin{equation}
    (\gbar(\xi)')_{d}(y)
    = \psi'_{d,g,y}(\xi)
    \le \psi'_{d,g,y}(x) + \frc\ep2
    = d(g(x),y) + \frc\ep2
    \le \frc\ep2 + \frc\ep2 = \ep,
  \end{equation}
  as desired.

  Finally, suppose that $d\le d'$ are $Y$-metrics.  Then clearly
  $\psi'_{d,g,y} \le \psi'_{d',g,y}$, so by \autoref{thm:becauchy-le} we
  have $A_{d'} \subseteq A_d$ and $\gbar(\xi)'_d \le \gbar(\xi)'_{d'}$.

  Now \autoref{thm:cauchy-smalldet} gives us a (unique) Cauchy point of
  $Y$, which we may denote $\gbar(\xi)$, with the property that
  $\gbar(\xi)_d(y) = \psi_{d,g,y}(\xi)$ whenever the latter is $<1$.
  Since $Y$ is complete, this Cauchy point must be represented by some
  actual point; choose such a point and denote it also by
  $\gbar(\xi)$.  Thus for any $y\in Y$ and $Y$-metric $d$ we have
  \[ d(\gbar(\xi),y) = \psi_{d,g,y}(\xi) \]
  if the latter is $<1$.  In particular, for any $x\in X$, since
  $\psi_{d,g,g(x)}(\xchk) = d(g(x),g(x)) = 0 < 1$, we have
  \[ d(\gbar(\xchk),g(x)) = \psi_{d,g,y}(\xchk) = 0
  \]
  for any $d$, and therefore $\gbar(\xchk) \approx g(x)$ as desired.

  We now show that $\gbar$ is continuous.  Let $\xi\in\Xhat$, let $d$
  be a $Y$-metric, and let $\ep>0$.  Choose a $y\in Y$ which satisfies
  $y\in A_d$ (i.e. $\psi_{d,g,y}(\xi)<1$) and also
  $(\gbar(\xi)')_d(y)<\ep$.  Such a $y$ exists by the third condition
  of \autoref{thm:cauchy-smalldet}, which we proved above for
  $\gbar(\xi)'$.  Then since $y\in A_d$, we have $d(\gbar(\xi),y) =
  \psi_{d,g,y}(\xi) <\ep$.

  In particular, for any other $\ze\in\Xhat$, by
  \autoref{thm:becauchy-le}\ref{item:becle1} and the triangle
  inequality for $d$,
  \begin{align}
    \psi_{d,g,\gbar(\ze)}(\xi)
    &\le d(\gbar(\ze),y)+ \psi_{d,g,y}(\xi) \\
    &< d(\gbar(\ze),y) - d(y,\gbar(\xi)) + 2\ep\\
    &\le d(\gbar(\ze),\gbar(\xi)) + 2\ep.
  \end{align}
  Since this holds for any $\ep>0$, we have
  $\psi_{d,g,\gbar(\ze)}(\xi) \le d(\gbar(\ze),\gbar(\xi))$.  Since we
  similarly have
  \begin{align}
    \psi_{d,g,\gbar(\ze)}(\xi)
    &\ge d(\gbar(\ze),y) - \psi_{d,g,y}(\xi) \\
    &> d(\gbar(\ze),y) + d(y,\gbar(\xi)) - 2\ep\\
    &\ge d(\gbar(\ze),\gbar(\xi)) - 2\ep.
  \end{align}
  for all $\ep>0$, we must in fact have $\psi_{d,g,\gbar(\ze)}(\xi) =
  d(\gbar(\ze),\gbar(\xi))$.

  Let $d_X$ be the metric on $X$ in $\cG_\be$ induced by the function
  $\psi_{d,g,\gbar(\xi)}$, and let $\de=\ep$.  Then if
  $d_X(\xi,\ze)<\de$, we have
  \[ \Big|
  \psi_{d,g,\gbar(\xi)}(\xi) - \psi_{d,g,\gbar(\xi)}(\ze) \Big|
  \;=\;
  \Big| d(\gbar(\xi),\gbar(\xi)) - d(\gbar(\ze),\gbar(\xi)) \Big|
  \; =\;
  d(\gbar(\ze),\gbar(\xi)) \;<\;\ep
  \]
  which is exactly what we need for continuity of $\gbar$.
  
  Finally, if $h:\Xhat\to Y$ is continuous and satisfies $h(\xchk)
  \approx g(x)$ for all $x$, then we have $h(\xchk) \approx
  \gbar(\xchk)$ for all $x$.  Fix $\xi\in\Xhat$, a $Y$-metric $d_Y$,
  and $\ep>0$; we want to show $d_Y(h(\xi),\gbar(\xi))<\ep$.  Since
  $h$ is continuous, there exists an $\Xhat$-metric $d_1$ and
  $\de_1>0$ such that (in particular) if $d_1(\xi,\xchk)<\de_1$, then
  $d_Y(h(\xi),h(\xchk))<\frac\ep2$.  Similarly, since $\gbar$ is continuous,
  there exists an $\Xhat$-metric $d_2$ and $\de_2>0$ such that if
  $d_2(\xi,\xchk)<\de_2$, then $d_Y(\gbar(\xi),\gbar(\xchk))<\frac\ep2$.

  Pick $d_3 \ge \max(d_1,d_2)$ and $\de_3 = \min(\de_1,\de_2)$.  By
  locatedness of $\xi$, there exists $x\in X$ such that
  $d_3(\xi,\xchk)<\de_3$, hence $d_Y(h(\xi),h(\xchk))<\frac\ep2$ and
  $d_Y(\gbar(\xi),\gbar(\xchk))<\frac\ep2$.
  Since $d_Y(\gbar(\xchk),h(\xchk))=0$, the triangle inequality
  implies $d_Y(h(\xi),\gbar(\xi))<\ep$ as desired.
\end{proof}

If $Y$ is itself a (separated) compactification of $X$, this says that
we have a unique map from $\be X$ to $Y$ fixing the image of $X$.
That is, every point at infinity in $\be X$ becomes a point in $Y$,
but one point in $Y$ may come from multiple points in $\be X$.  This
is the sense in which $\be X$ is the \emph{best} compactification of
$X$.  (The technical term is \emph{initial}.  One also says that $\be$
is the \emph{left adjoint} to the inclusion of compact separated gauge
spaces in all gauge spaces.)


\section{Exercises on the Stone-\v Cech compactification}
\label{sec:ex-stonecech}

\begin{ex}
  Prove that if $X$ is already compact and separated, then $\be X$ is
  topologically isomorphic to $X$.
\end{ex}

\begin{ex}\label{ex:naturality}
  For any map $g:X\to Y$, the composite $X\xto{g} Y \to \be Y$ induces
  a map $\be g: \be X \to \be Y$.  Prove that the composites $X\to \be
  X \to \be Y$ and $X\to Y\to \be Y$ are equal.  (This is called
  \emph{naturality}.)
\end{ex}

\begin{exstar}
  Use \autoref{ex:cplt-xtn} to give an alternative proof of
  \autoref{thm:extension}.
\end{exstar}

\begin{ex}\label{ex:e-unif}
  Let $X$ be a gauge space, and define a new gauge $\cG_e$ similarly to
  $\cG_\be$, but using continuous functions $X\to\lR^n$ instead of
  $X\to [0,1]^n$.  Prove that:
  \begin{enumerate}
  \item $\cG_e$ is topologically equivalent to the original gauge.
  \item If $\cG_e$ is totally bounded, then $X$ is pseudocompact
    (see \autoref{ex:pscpt}).
  \end{enumerate}
\end{ex}

\begin{ex}
  Suppose that $X$ has a \emph{discrete} metric.  Prove that the
  points of $\be X$ can be identified with \textbf{ultrafilters} on
  $X$; that is, subsets $\cF\subseteq \cP X$ such that
  \begin{itemize}
  \item $X\in\cF$ and $\emptyset\not\in\cF$,
  \item If $A,B\in\cF$ then $A\cap B\in\cF$,
  \item If $A\in\cF$ and $A\subseteq B$, then $B\in \cF$, and
  \item For all $A\subseteq X$, either $A\in\cF$ or $\neg A\in\cF$.
  \end{itemize}
\end{ex}


\begin{exstar}
  Prove that no point of $\be \lN\setminus \lN$ (where \lN\ has the
  discrete metric) is the limit of a sequence of points of \lN.
  Conclude that $\be\lN$ is not metrizable (that is, not
  topologically equivalent to a space with a single metric).
\end{exstar}

\begin{exstar}
  Let $X$ be a gauge space, $\xi$ a point of $\be X \setminus X$, and
  let $Y = \be X \setminus \{\xi\}$.  Prove that $\be Y$ is
  topologically isomorphic to $\be X$.
\end{exstar}

\begin{exstar}
  Let \Om\ be an uncountable well-ordered set (that is, it has a total
  order $<$ and any subset of it has a least element) such that for
  any $a\in\Om$, the set $\{b\in\Om\mid b<a\}$ is countable.  For any
  $a\in\Om$ define
  \[d_a(x,y) =
  \begin{cases}
    0 & x=y \text{ or } (x>a \text{ and } y>a)\\
    1 & \text{otherwise.}
  \end{cases}\]
  \begin{enumerate}
  \item Verify that the set of all these metrics, $\{d_a \mid a\in
    \Om\}$, is a gauge on \Om.
  \item Prove that any continuous function $f\maps \Om\to [0,1]$ is
    eventually constant, i.e.\ there is an $a\in\Om$ such that
    $f(x)=f(a)$ whenever $x\ge a$.
  \item Prove that $\be\Om$ contains exactly one point not in \Om.
  \end{enumerate}
\end{exstar}

\begin{exstar}
  Suppose $X$ is equipped with a topology as in
  \autoref{ex:top-space}.  Classically, $X$ is said to be
  \textbf{completely regular} if for any point $x\in X$ and open
  subset $U\subseteq X$ such that $x\in U$, there exists a continuous
  function $\ph:X\to [0,1]$ (in the sense of \autoref{ex:top-cts})
  such that $\ph(x) = 1$ and $\ph(X\setminus U) = \{0\}$.  Prove that
  any such topology is induced by some gauge on $X$.
\end{exstar}

\begin{exstar}
  Prove that every gauge space is proximally equivalent to a totally
  bounded one.  (This is a sort of converse to \autoref{ex:tbddproxunif}.)
\end{exstar}

\begin{exstar}
  A \textbf{proximity} on a set $X$ is a relation $\approx$ between
  subsets of $X$ such that
  \begin{itemize}
  \item $A\not\approx \emptyset$ for all $A$ \emph{(nontriviality)}
  \item $A\cap B \neq \emptyset$ implies $A\approx B$ \emph{(reflexivity)}
  \item $A\approx B\iff B\approx A$ \emph{(symmetry)}
  \item If $A\not\approx B$ then there is a $C$ with $B\not\approx C$
    and $A\not\approx (X\setminus C)$ \emph{(transitivity)}
  \item $A\approx (B\cup C)\iff A\approx B$ or $A\approx C$ \emph{(filteredness)}
  \end{itemize}
  Prove that:
  \begin{enumerate}
  \item The relation $\approx$ induced by a gauge is a proximity.
  \item Every proximity is induced by a gauge.
  \end{enumerate}
\end{exstar}

\begin{exstar}
  (This exercise continues \autoref{ex:tbdd-cauchy}.)  Suppose $X$ is
  a totally bounded gauge space and $\sF$ is a nonempty collection of
  subsets of $X$ such that
  \begin{itemize}
  \item $\emptyset\notin\sF$ \emph{(nontriviality)}.
  \item $(A\cup B)\in\sF$ $\iff$ $A\in\sF$ or $B\in\sF$ \emph{(filteredness)}.
  \item If $A\notin\sF$, then there is a $B\in\sF$ such that
    $A\not\approx B$ \emph{(transitivity)}.
  \item If $A\in \sF$ and $B\in \sF$, then $A\approx B$ \emph{(locatedness)}.
  \end{itemize}
  Prove that there is a unique Cauchy point $\xi$ of $X$ such that
  $\sF = \setof{A\subseteq X | \xi \approx A}$.  Thus, completion of
  totally bounded gauge spaces (hence, compactifications) can
  equivalently be expressed using proximities.
\end{exstar}

\begin{exstar}
  A \textbf{uniformity} on a set $X$ is a nonempty collection of
  subsets of $X\times X$, called \emph{entourages}, such that
  \begin{itemize}
  \item If $U$ is an entourage and $U\subseteq V$, then $V$ is an
    entourage \emph{(saturation)}.
  \item If $U$ and $V$ are entourages, then so is $U\cap V$
    \emph{(filteredness)}.
  \item If $U$ is an entourage, then $(x,x)\in U$ for all $x\in X$
    \emph{(reflexivity)}.
  \item If $U$ is an entourage, then so is $U^{-1} = \setof{(y,x) |
      (x,y) \in U }$ \emph{(symmetry)}.
  \item If $U$ is an entourage, then there exists an entourage $V$
    such that
    \begin{equation}
      V\circ V = \setof{(x,z) | (x,y)\in V \;\text{and}\; (y,z)\in V}
    \end{equation}
    is a subset of $U$ \emph{(transitivity)}.
  \end{itemize}
  A function $f:X\to Y$ between sets equipped with uniformities
  (called \emph{uniform spaces}) is said to be \textbf{uniformly
    continuous} if whenever $U$ is an entourage of $Y$, then $(f\times
  f)^{-1}(U)$ is an entourage of $X$.  Prove that:
  \begin{enumerate}
  \item Every gauge induces a uniformity.
  \item A function $f:X\to Y$ between gauge spaces is uniformly
    continuous in our sense if and only if it is uniformly continuous
    in the entourage sense.
  \item Every uniformity is induced by a gauge.\label{item:unif}
  \item Two gauges on a set induce the same uniformity if and only if
    they are uniformly equivalent.
  \end{enumerate}
\end{exstar}

\section{Constructing points in $\be X$}
\label{sec:points-in-betaX}

One thing which is still missing is a proof that every point at
infinity in some other compactification $Y$ must come from \emph{at
  least one} point at infinity in $\be X$.  In other words, we want to
be able to get any other compactification by squashing together some
points in $\be X$.

The reason this is hard to prove is that even with
\autoref{thm:cauchy-beta}, the points of $\be X$ are still pretty
amorphous beasts.  However, there is a fairly easy way to produce them
(even if the result is not all that explicit).  The idea is to start
with an approximation and gradually refine it.

\begin{defn}
  An \textbf{evaluation datum} on a gauge space $X$ is an operation
  $J$ assigning to each continuous $\ph:X\to [0,1]$ a nonempty
  \emph{closed interval} $J(\ph) \subseteq [0,1]$, such that
  \begin{quote}
    For any finite family $(\ph_i)_{1\le i\le n}$ and every $\ep>0$,
    there exists a point $x\in X$ such that
    $d(\ph_i(x),J(\ph_i))<\ep$ for all $i$.
  \end{quote}
\end{defn}

By \autoref{thm:cauchy-beta}, a Cauchy point of $(X,\cG_\be)$ is
precisely an evaluation datum such that each interval $J(\ph)$ has
length zero (i.e.\ is of the form $[\ph(\xi),\ph(\xi)]$).  The next lemma
says that we can ``improve'' any evaluation datum to become more like
a Cauchy point.

\begin{lem}\label{thm:eval-improve}
  Let $J$ be an evaluation datum, $\ph:X\to [0,1]$ continuous, and
  suppose $J(\ph) = I_1 \cup I_2$ is the union of two closed
  subintervals.  Define $J_1$ to be $J$ except that $J_1(\ph) = I_1$, and
  similarly $J_2$ to be $J$ except that $J_2(\ph)= I_2$.  Then either
  $J_1$ or $J_2$ is an evaluation datum.
\end{lem}
\begin{proof}
  Suppose that neither is.  Then we have an $\ep_1>0$ and a finite
  family $(\psi_i)_{1\le i\le n}$ forming a counterexample for $J_1$,
  and similarly $\ep_2>0$ and $(\chi_j)_{1\le j\le m}$ forming a
  counterexample for $J_2$.  Since $J$ is an evaluation datum, we must
  have $\psi_i = \ph$ for some $i$ and $\chi_j = \ph$ for some $j$.
  But now $\ep = \min(\ep_1,\ep_2)$ and the union of the $\psi$'s and
  $\chi$'s is a counterexample for $J$, since a point that is within
  $\ep$ of $J(\ph)$ must either be within $\ep$ of $I_1$ or $I_2$.
\end{proof}

Using this, we can produce a Cauchy point refining any evaluation datum.

\begin{thm}\label{thm:making-points}
  For any evaluation datum $J$, there exists a point $\xi\in \be X$ such
  that $\ph(\xi)\in J(\ph)$ for all $\ph$.
\end{thm}
\begin{proof}
  Consider the partially ordered set of evaluation data, where $J_1
  \le J_2$ is defined to mean $J_2(\ph) \subseteq J_1(\ph)$ for all
  $\ph$.  \autoref{thm:eval-improve} implies that if $J$ is an
  evaluation datum that is not a $\beta$-Cauchy point, then there is
  an evaluation datum $J'$ with $J<J'$.  Thus, a maximal element of
  this poset must be a $\beta$-Cauchy point.

  We will apply Zorn's Lemma to show that a maximal element exists,
  and indeed a maximal element exists above any element; this will
  prove the theorem.  Thus, we must show that chains of evaluation
  data have upper bounds.  Let $(J_a)_{a\in A}$ be a chain, i.e.\ a
  set of evaluation data such that $\le$ restricted to it is a total
  order, and define
  \[ J(\ph) = \bigcap_{a\in A} J_a(\ph).
  \]
  Evidently the lower bound of $J(\ph)$ is the supremum of the lower
  bounds of the $J_a(\ph)$ and likewise its upper bound is the infimum
  of their upper bounds.

  Clearly $J\ge J_a$ for all $a$, so it remains to show that $J$ is an
  evaluation datum.  Let $(\ph_i)_{1\le i\le n}$ and $\ep>0$ be given.
  Since $(J_a)_a$ is a chain (or more precisely, filtered), for each
  $i$ there exists some $a_i$ such that the length of $J_{a_i}(\ph)$
  is no more than $\frc\ep2$ greater than the length of $J_a(\ph)$.
  Again, since $(J_a)_a$ is a chain (or filtered) and there are only
  finitely many $i$'s, there exists some $b\in A$ which is above all of
  these $a_i$'s.  But since $J_b$ is an evaluation datum, we have an
  $x\in X$ such that $d(\ph_i(x), J_b(\ph)) < \frc\ep2$ for all $i$,
  and hence $d(\ph_i(x), J(\ph)) < \ep$.  Thus $J$ is an evaluation
  datum, completing the proof.
\end{proof}

Now we can prove the missing theorem.

\begin{thm}\label{thm:ext-surj}
  Suppose $Y$ is compact and separated, and $g:X\to Y$ is a continuous
  map such that for all $y\in Y$ we have $g(X) \approx y$.  Then the
  unique extension $\gbar:\be X \to Y$ is surjective.
\end{thm}
\begin{proof}
  Given $y\in Y$, define an evaluation datum $J$ on $X$ as follows.
  For each $Y$-metric $d$ and $\ep>0$, let $C_{d,\ep} = \setof{x\in X
    | d(g(x),y) < \ep }$.  Since $g(X)\approx y$ by assumption,
  $C_{d,\ep}$ is always nonempty.  For $\ph:X\to [0,1]$ continuous, let
  \begin{align}
    A(\ph) &= \sup_{d,\ep} \; \inf_{x\in C_{d,\ep}} \ph(x)\\
    B(\ph) &= \inf_{d,\ep} \; \sup_{x\in C_{d,\ep}} \ph(x).
  \end{align}
  and set $J(\ph) = [A(\ph),B(\ph)]$.  To show that $J$ is an evaluation
  datum, let $(\ph_i)_{1\le i\le n}$ and $\ep'>0$ be given.  Since
  $Y$-metrics are filtered and positive real numbers admit minima, we
  can then choose a single $Y$-metric $d$ and an $\ep>0$ such that for
  all $i$, we have
  \begin{equation}
    \left|A({\ph_i}) - \inf_{x\in C_{d,\ep}} \ph_i(x) \right| < \ep'\qquad\text{and}\qquad
    \left|B({\ph_i}) - \sup_{x\in C_{d,\ep}} \ph_i(x) \right| < \ep'
  \end{equation}
  Thus, for \emph{any} $x\in C_{d,\ep}$, since $\ph_i(x)$ lies between
  $\inf_{x\in C_{d,\ep}} \ph_i(x)$ and $\sup_{x\in C_{d,\ep}}
  \ph_i(x)$ for all $i$, it must also lie within $\ep'$ of $J(\ph_i)$.

  As in the proof of \autoref{thm:extension}, for any $Y$-metric $d$
  we define $\psi_{d,g,y}(x) = d(g(x),y)$.  Note that then for any
  $\ep>0$, we have
  \begin{equation}
    \sup_{x\in C_{d,\ep}} \psi_{d,g,y}(x) \le \ep.
  \end{equation}
  Therefore $B({\psi_{d,g,y}}) = 0$, and so $J(\psi_{d,g,y}) = [0,0]$.

  Now we apply \autoref{thm:making-points} to obtain a point
  $\xi\in\be X$ such that $\ph(\xi) \in J(\ph)$ for all $\ph$.  In
  particular, for any $Y$-metric $d$ we have $\psi_{d,g,y}(\xi) = 0$.
  Since $0<1$, by the proof of \autoref{thm:extension} we also have
  \[ d(\gbar(\xi),y) = \psi_{d,g,y}(\xi) = 0. \] Since this is true
  for any $d$, and $Y$ is separated, $\gbar(\xi) = y$.
\end{proof}

Thus, for instance, $\be (\lR^2)$ contains points at infinity which map
onto all the points at infinity in the projective plane.  In fact,
each point at infinity in the projective plane is the image of
\emph{many} points at infinity in $\be(\lR^2)$.

\begin{eg}
  For any fixed $y_0\in\lR$, define an evaluation datum $J_y$ on $\lR^2$ by setting
  \begin{align}
    A_{y_0}(\ph) &= \sup_{z\in\lR}\; \inf_{x>z}\; \ph(x,y_0)\\
    B_{y_0}(\ph) &= \inf_{z\in\lR}\; \sup_{x>z}\; \ph(x,y_0)\\
    J_{y_0}(\ph) &= [A_{y_0}(\ph),B_{y_0}(\ph)].
  \end{align}
  Since finite sets of real numbers have maxima, the same argument as
  in the proof of \autoref{thm:ext-surj} shows that $J_{y_0}$ is an
  evaluation datum.  Thus, it is refined by some point
  $\xi_{y_0}\in\be (\lR^2)$.  It is easy to see that $\xi_{y_0} \neq
  \xi_{y_1}$ if $y_0\neq y_1$ (consider a continuous function which is
  constant at zero on the line $y=y_0$, but constant at $1$ on the
  line $y=y_1$).

  Thus, although in the projective plane all pairs of parallel lines
  meet at exactly one point, in $\be(\lR^2)$, every line contains points
  at infinity that are \emph{not} shared by any line parallel to it.
  Also, we can do the same thing with $x<z$ rather than $x>z$ and
  obtain different points, so there are different points at infinity
  in opposite directions on the same line (just as in the
  lower-hemisphere compactification).
\end{eg}

The numbers $A_{y_0}(\ph)$ and $B_{y_0}(\ph)$ defined in the previous example
are sometimes called the \emph{limit inferior} and the \emph{limit
  superior}, respectively, of $\ph(x,y_0)$ as $x\to\infty$, and written
$\liminf_{x\to\infty} \ph(x,y_0)$ and $\limsup_{x\to\infty} \ph(x,y_0)$.
The similar definitions in \autoref{thm:ext-surj} could also be called
limits inferior and superior of a sort.

\section{Final exercises}
\label{sec:ex-constr}

\begin{ex}
  Prove that if the canonical map $X\to \be X$ is surjective, then $X$
  is complete.
\end{ex}

\begin{exstar}
  Prove that if $X\to \be X$ is surjective, then $X$ is totally
  bounded.
\end{exstar}

\begin{ex}
  Prove that $X$ is compact if and only if $X\to \be X$ is
  surjective.  Conclude that compactness is a topological property.
\end{ex}

\begin{ex}
  Prove that $X$ is compact if and only if the gauge $\cG_\be$ is complete.
\end{ex}

\begin{exstar}\label{ex:cpt-prox}
  Prove that in a compact gauge space, we have $A\approx B$ if and
  only if there exists a point $x$ with $x\approx A$ and $x\approx B$.
\end{exstar}

\begin{ex}\label{ex:surj-cpt}
  Prove that if $g:X\to Y$ is a continuous surjection with $X$
  compact, then $Y$ is also compact.
\end{ex}

\begin{ex}
  Prove that any compact gauge space is pseudocompact.
\end{ex}

\begin{ex}
  Prove that any continuous map with compact domain is in fact
  \emph{uniformly} continuous.  Conclude that if two compact gauge
  spaces are topologically isomorphic, then they are uniformly
  isomorphic.
\end{ex}

\bibliographystyle{plain}
\bibliography{all}

\end{document}